\documentclass[11pt]{article}

\usepackage[T1]{fontenc}
\usepackage[english]{babel}
\usepackage{authblk}

\usepackage{microtype}
\usepackage{amsmath,amssymb,amsfonts,amsthm,mathtools}
\usepackage{geometry}
\usepackage{hyperref}
\hypersetup{hidelinks}
\usepackage[shortlabels]{enumitem}
\usepackage[numbers,sort&compress]{natbib}

\usepackage{color}

\newtheorem{theorem}{Theorem}[section]
\newtheorem{lemma}[theorem]{Lemma}
\newtheorem{proposition}[theorem]{Proposition}
\newtheorem{claim}[theorem]{Claim}
\newtheorem{corollary}[theorem]{Corollary}
\newtheorem{conjecture}[theorem]{Conjecture}
\newtheorem{remark}[theorem]{Remark}

\newcommand{\eps}{\varepsilon}
\newcommand{\VC}{\operatorname{VCdim}}
\newcommand{\eVC}{\operatorname{eVCdim}}

\newcommand{\ceil}[1]{\lceil #1\rceil}

\title{A Single-Exponential Erd\H{o}s--Hajnal Bound for Graphs of Bounded VC-Dimension}

\author[1]{\textbf{Shuang Sun}\thanks{Email: \href{mailto:chocolatesun@sjtu.edu.cn}{chocolatesun@sjtu.edu.cn}.}}
\author[2]{\textbf{Yan Wang}\thanks{Email: \href{mailto:yan.w@sjtu.edu.cn}{yan.w@sjtu.edu.cn}.}}
\author[3]{\textbf{Jiasheng Zeng}\thanks{Email: \href{mailto:jasonzeng@mail.ustc.edu.cn}{jasonzeng@mail.ustc.edu.cn}.}}

\affil[1]{\footnotesize School of Mathematical Sciences, Shanghai Jiao Tong University, Shanghai 200240, China}
\affil[2]{\footnotesize School of Mathematical Sciences, CMA-Shanghai, Shanghai Jiao Tong University, Shanghai 200240, China}
\affil[3]{\footnotesize Department of Mathematics, Hong Kong University of Science and Technology, Clear Water Bay, Hong Kong}

\date{}
\begin{document}
\maketitle
\begin{abstract}

A homogeneous set in a graph is a clique or a stable set. 
The Erd\H{o}s--Hajnal conjecture states that, for every graph $H$, there
exists $c>0$ such that every $H$-free graph on $n$ vertices has a homogeneous set of size at least $n^c$.
Nguyen, Scott and
Seymour proved that for every $d>0$, graphs of VC-dimension at most $d$
have the Erd\H{o}s--Hajnal property, confirming a conjecture of Fox, Pach and Suk. In particular, they showed that every such $n$-vertex graph contains a homogeneous set of size at least $n^{\eta_d}$ for some $\eta_d\ge 2^{-2^{O(d)}}$.
In this paper, we give a sharper quantitative bound on the homogeneous sets in graphs of VC-dimension at most $d$, showing that one may take
$
        \eta_d\ge (Cd)^{-d},
$
where $C$ is an absolute constant.  Equivalently, every graph $G$ of VC-dimension at most $d$ satisfies
\[
        \max\{\omega(G),\alpha(G)\}\ge |G|^{(Cd)^{-d}}.
\]
Our proof refines the iterative sparsification method of Nguyen, Scott and Seymour. 
The main enhancement is to apply the VC-dimension assumption directly, which gives
a more efficient induction and thus improves the dependence on $d$.
We also derive quantitative consequences for polynomial R\"odl subgraphs, hypergraph Ramsey bounds under bounded VC-dimension, induced-free and viral formulations, tournaments, NIP and semi-algebraic graphs, Boolean combinations of relations of bounded VC-dimension, graphs whose adjacency matrices have bounded rank, graphs of bounded sign-rank, and graphs defined by dot-product threshold representations.
\end{abstract}

\section{Introduction}
All (hyper)graphs in this paper are finite and simple. A graph $H$ is an
induced subgraph of a graph $G$ if $H$ can be obtained from $G$ by removing
vertices. A class $\mathcal C$ of graphs is hereditary if it is closed under
taking induced subgraphs and under isomorphism; and a hereditary class
$\mathcal C$ is proper if it is not the class of all graphs.
The Erd\H{o}s--Hajnal conjecture is a far-reaching strengthening of Ramsey's theorem for hereditary graph classes. We say that a class $\mathcal C$ of graphs has the \emph{Erd\H{o}s--Hajnal property} if there exists $c=c(\mathcal C)>0$ such that every $n$-vertex graph $G\in\mathcal C$ contains a clique or a stable set of size at least $n^c$. Equivalently, for every fixed graph $H$, the conjecture asserts that the hereditary class of graphs with no induced copy of $H$ has this property~\cite{ErdosHajnal1977,ErdosHajnal1989}.

\begin{conjecture}
    Every proper hereditary class of graphs has the Erd\H{o}s--Hajnal property.
\end{conjecture}

This conjecture has become one of the central problems in induced Ramsey theory; see a survey of Chudnovsky~\cite{Chudnovsky2014}. A basic structural principle was established by Alon, Pach and Solymosi, who proved that the Erd\H{o}s--Hajnal property is preserved under substitution operation, which reduces the general problem to prime graphs~\cite{AlonPachSolymosi2001}. Several difficult prime or near-prime cases have since been settled. Chudnovsky and Safra proved the conjecture for bull-free graphs \cite{ChudnovskySafra2008}. Chudnovsky, Fox, Scott, Seymour and Spirkl obtained a strengthened lower bound for graphs with no induced $5$-cycle~\cite{ChudnovskyFoxScottSeymourSpirkl2019}, and Chudnovsky, Scott, Seymour and Spirkl later proved the full Erd\H{o}s--Hajnal conjecture for this class~\cite{ChudnovskyScottSeymourSpirkl2023}. Nguyen, Scott and Seymour proved the conjecture for $P_5$-free graphs \cite{NguyenScottSeymour2026VII}; together with the $C_5$-free case, this completes the conjecture for all graphs on at most five vertices. 
Very recently, Huang, Ju and Zhou proved the Erd\H{o}s--Hajnal property for two six-vertex graphs, the $E$-graph obtained from $P_5$ by adding a pendant edge to the middle vertex and the Bird graph obtained from the bull by adding a pendant edge to one horn~\cite{HuangJuZhou2026}.

Other progress comes from geometric and low-complexity intersection patterns. A graph class is said to have the \textit{strong Erd\H{o}s--Hajnal property} if there exists $\delta>0$ such that every $n$-vertex graph in the class contains two disjoint vertex sets $A$ and $B$, each of size at least $\delta n$, that are complete or anticomplete to each other in $G$.  Pach and Solymosi proved the strong Erd\H{o}s--Hajnal property for intersection graphs of segments \cite{PachSolymosi2001}, Alon, Pach, Pinchasi, Radoi\v{c}i\'c and Sharir proved an Erd\H{o}s--Hajnal type result for semi-algebraic intersection graphs~\cite{AlonPachPinchasiRadoicicSharir2005}, and Fox, Pach and T\'oth obtained related results for intersection patterns of curves \cite{FoxPachToth2011}.
Tomon then proved that every string graph has the Erd\H{o}s--Hajnal property~\cite{Tomon2024}. Recent work has also produced new general mechanisms for induced Ramsey problems. Fox, Nguyen, Scott and Seymour proved sparse-versus-dense density statements for cographs~\cite{FoxNguyenScottSeymour2025}; Nguyen, Scott and Seymour found infinitely many new prime graphs with the Erd\H{o}s--Hajnal property~\cite{NguyenScottSeymour2023IV}, and proved that all paths satisfy near-polynomial Erd\H{o}s--Hajnal bounds \cite{NguyenScottSeymour2023V}. In a different direction, Buci\'c, Nguyen, Scott and Seymour improved the general lower bound for every fixed forbidden induced subgraph $H$~\cite{BucicNguyenScottSeymour2024}, while Buci\'c, Fox and Pham proved that the Erd\H{o}s--Hajnal conjecture is equivalent to polynomial versions of the R\"odl and Nikiforov conjectures \cite{BucicFoxPham2024}.

A particularly rich source of graphs satisfying the Erd\H{o}s--Hajnal property arises from low-complexity set-system conditions, most notably bounded Vapnik--Chervonenkis dimension (VC-dimension for short).
For a family $\mathcal{F}$ of subsets of a set $X$, a set $S\subseteq X$ is \emph{shattered} by $\mathcal{F}$ if for every $A\subseteq S$ there exists $F\in\mathcal{F}$ with $F\cap S=A$. The \emph{VC-dimension} of $\mathcal{F}$ is the largest size of a shattered set~\cite{VapnikChervonenkis1971}. 
If $S\subseteq V(G)$ and $v\in V(G)$, we call $N_G(v)\cap S$ the
\emph{trace} of $v$ on $S$.
For a graph $G$, its VC-dimension $\operatorname{VCdim}(G)$ is the VC-dimension of the neighbourhood family $\{N_G(v):v\in V(G)\}$, where $N_G(v)$ denotes the set of all neighbours of $v$ in $G$, not including
$v$ itself.
This definition was linked to model theory by Laskowski's characterization of NIP formulas via finite VC-dimension~\cite{Laskowski1992}.
When a graph is viewed through the set system of its vertex neighbourhoods, bounded VC-dimension becomes a concrete combinatorial way of saying that the graph cannot encode arbitrary set systems, and this viewpoint has led to stronger regularity lemmas and structural results for hereditary graph classes~\cite{AlonFischerNewman2007,LovaszSzegedy2010,AlonBaloghBollobasMorris2011}. The same principle also manifests in extremal graph theory: For graphs and hypergraphs of bounded description complexity, including semi-algebraic and distal relations, one obtains much sharper regularity lemmas and Zarankiewicz-type bounds than in the general setting~\cite{FoxPachSuk2016,FoxPachShefferSukZahl2017,ChernikovStarchenko2018Distal,ChernikovGalvinStarchenko2020,Do2018}, and the Zarankiewicz problem has been studied directly for bipartite graphs of bounded VC-dimension \cite{JanzerPohoata2024}. A parallel line of work in extremal set theory, originating with the Frankl--Pach theorem on uniform set systems of bounded VC-dimension \cite{FranklPach1984}, has recently been improved by Ge, Xu, Yip, Zhang and Zhao \cite{GeXuYipZhangZhao2026}, by Chao, Xu, Yip and Zhang \cite{ChaoXuYipZhang2025}, and by Yang and Yu \cite{YangYu2025}, suggesting that bounded VC-dimension is not only a regularity condition but also a robust extremal hypothesis connecting incidence geometry, Tur\'an-type problems and uniform set systems.

In induced Ramsey theory, bounded VC-dimension has been a subject of extensive investigation. Fox, Pach and Suk proved strong near-polynomial Erd\H{o}s--Hajnal type bounds for graphs of bounded VC-dimension~\cite{FoxPachSuk2019}, building on a broader line of induced Ramsey-type and induced-density results~\cite{FoxSudakov2008,BucicNguyenScottSeymour2024,NguyenScottSeymour2026VII}. This direction culminated in the theorem of Nguyen, Scott and Seymour, who proved that for every fixed $d$, the class of graphs with $\operatorname{VCdim}(G)\le d$ has the Erd\H{o}s--Hajnal property and, in fact, the polynomial R\"odl property \cite{NguyenScottSeymour2025}. Let $\eta_d$ be the supremum of all $\eta>0$ such that every graph $G$ with $\operatorname{VCdim}(G)\le d$ contains a clique or a stable set of size at least $|G|^\eta$; their proof yields the explicit bound $\eta_d \ge 2^{-2^{O(d)}}$.

We write $\omega(G)$ for the clique number of $G$, $\alpha(G)$ for the
stability number of $G$, and $|G|$ for $|V(G)|$.
The main result of this paper is to give a quantitative improvement of $\eta_d$ from the reciprocal of double exponential to single exponential.

\begin{theorem}\label{thm:main}
There exists $h\ge 1$ such that, for every integer $d\ge 1$ and every graph $G$ with $\VC(G)\le d$,
\[
        \max\{\omega(G),\alpha(G)\}\ge |G|^{(hd)^{-d}}.
\]
Equivalently,
\[
        \eta_d\ge (hd)^{-d}.
\]
\end{theorem}


We also record several consequences of Theorem~\ref{thm:main}. They fall into two groups. The first group consists of consequences obtained from standard reductions once the Erd\H{o}s--Hajnal exponent is made explicit.
Here an \emph{Erd\H{o}s--Hajnal exponent} for a graph class $\mathcal C$ is a
positive number $c>0$ such that every $n$-vertex graph $G\in\mathcal C$ has a clique
or stable set of size at least $n^c$.
The same sparsification argument gives a polynomial R\"odl exponent at most $(Cd)^d$. Substituting the graph-level estimate into the Erd\H{o}s--Rado greedy argument gives explicit Ramsey bounds for hypergraphs whose relevant trace families have bounded VC-dimension, including the multicolor variant stated below. The induced-free and viral formulations of Nguyen, Scott and Seymour likewise become quantitative after applying the available bootstrapping theorems \cite{GishbolinerShapira2023,BucicFoxPham2024}. The tournament version follows by applying the graph theorem to the in-neighbourhood set system of the tournament. 
The second group consists of applications that give concrete criteria for verifying the bounded VC-dimension hypothesis in standard settings. 
We now explain how the required VC-dimension bounds are obtained in these applications. Boolean combinations of relations of bounded VC-dimension again have bounded VC-dimension, by the Sauer--Shelah lemma~\cite{VapnikChervonenkis1971,Sauer1972,Shelah1972}. NIP definable relations give a model-theoretic source of finite VC-dimension \cite{Simon2015,ChernikovStarchenkoThomas2021,ChernikovStarchenko2021NIP}. For semi-algebraic graphs, the required bounds follow from sign-pattern estimates, especially Warren's theorem and its standard incidence-geometric consequences \cite{Warren1968,ConlonFoxPachSudakovSuk2014,FoxPachShefferSukZahl2017}. Graphs whose adjacency matrices have bounded rank admit a direct linear-algebraic VC-dimension bound. Finally, graphs of bounded sign-rank, as well as graphs defined by thresholding dot products of vectors, have bounded VC-dimension through standard sign-rank estimates~\cite{Forster2002,AlonMoranYehudayoff2017}. Thus each of these settings falls under the main theorem and satisfies an explicit Erd\H{o}s--Hajnal
bound.

We briefly indicate the quantitative mechanism behind the proof. For $r\ge 0$, let $\mathcal C_r$ be the class of graphs with no bi-induced copy of the universal shattering bigraph $U_{r+1}$; thus $\operatorname{VCdim}(G)\le d$ implies $G\in\mathcal C_d$. Starting from a graph in $\mathcal C_r$, the regularity input produces a long $\varepsilon$-pure blockade, and the pattern graph records which pairs of blocks are almost complete. The key point is the mixed-block dimension drop. If a vertex is mixed on many blocks and the corresponding pattern graph
contained $U_r$, then the mixed vertex could be used as one more coordinate.
Indeed, in each block corresponding to a subset of $[r]$, we choose two
vertices, one adjacent and one nonadjacent to the mixed vertex. Purity then
allows us to choose vertices from the blocks corresponding to the
coordinates $1,\ldots,r$. These vertices form a bi-induced copy of
$U_{r+1}$, a contradiction. Hence the pattern graph lies in $\mathcal C_{r-1}$, so the induction hypothesis is used with the parameter $s_{r-1}$, where $s_{r-1}$ is the reciprocal of the corresponding Erd\H{o}s--Hajnal exponent. Iterating the resulting sparsification step gives, for every scale $\eta$, either an $\eta$-restricted induced subgraph of size at least $\eta^{O(r s_{r-1})}|G|$, or a complete or anticomplete blockade with polynomial width loss. A standard cograph induction converts this global alternative into the recursion $s_r \le C r s_{r-1}.$ Since $s_0=1$, this gives $s_r\le (Cr)^r$, and taking $r=d$ yields $\max\{\omega(G),\alpha(G)\}\ge |G|^{(Cd)^{-d}}.$

\medskip

The paper is organized as follows. In Section~\ref{sec:preliminaries} we introduce the local tools used in the proof.  We prove Theorem~\ref{thm:main} in Section~\ref{sec:main}. Section~\ref{sec:applications} records several quantitative applications of Theorem~\ref{thm:main}.We give a probabilistic obstruction for the Erd\H{o}s--Hajnal exponent and conclude in Section~\ref{sec:remarks}.

\section{Preliminaries and local lemmas}\label{sec:preliminaries}

A useful way to prove Erd\H{o}s--Hajnal-type results is to find large induced subgraphs whose edge structure becomes easier to control.
In this section we set up the local tools needed for this approach.  We
first use pure blockades to record approximate complete or anticomplete
relations between blocks, and then show that these approximate
patterns can be lifted to restricted induced subgraphs.  The key step is a dimension-drop argument: If a vertex is mixed on many blocks,
then the associated pattern graph belongs to a lower external
VC-dimension class.  This allows us to combine induction with iterative
sparsification, leading either to a more restricted induced subgraph or to
a large complete or anticomplete blockade.
For a graph $G$, write $V(G)$ and $E(G)$ for its vertex and edge sets,
$e(G)=|E(G)|$, $\Delta(G)$ for its maximum degree, and $\overline G$
for its complement.

\subsection{Blockades and restricted graphs}

A \emph{blockade} in a graph $G$ is a sequence
$
        \mathcal B=(B_1,\ldots,B_k)
$
of pairwise disjoint subsets of $V(G)$, called \textit{blocks}. 
For a positive integer $k$, write $[k]=\{1,\ldots,k\}$.
Its \emph{length} is $k$ and its \emph{width} is $\min_{i\in[k]} |B_i|$.  If its length is at least $\ell$ and its width is at least $w$, we call it an $(\ell,w)$-blockade.

For disjoint vertex sets $A,B\subseteq V(G)$ and $\eps>0$, say that $A$ is \textit{$\eps$-sparse} to $B$ in $G$ if every vertex of $A$ has at most $\eps |B|$ neighbours in $B$.  The pair $(A,B)$ is \textit{$\eps$-pure} if $A$ is $\eps$-sparse to $B$ and $B$ is $\eps$-sparse to $A$ either in $G$ or in $\overline G$.  A blockade is \textit{$\eps$-pure} if every two distinct blocks form an $\eps$-pure pair.
For disjoint sets $A,B\subseteq V(G)$, we say that $A$ is \textit{complete} to $B$
if every vertex of $A$ is adjacent to every vertex of $B$, and
\textit{anticomplete} to $B$ if there are no edges between $A$ and $B$. A blockade
$(B_1,\ldots,B_k)$ is \textit{complete}, respectively \textit{anticomplete}, if every two
distinct blocks are complete, respectively anticomplete, to each other.
We call a pair $(A,B)$ \textit{pure} if it is either complete or anticomplete.
A graph $G$ is \textit{$\eps$-sparse} if its maximum degree is at most $\eps |G|$, and is \textit{$\eps$-restricted} if either $G$ or $\overline G$ is $\eps$-sparse.

Let $\mathcal B=(B_1,\ldots,B_\ell)$ be an $\varepsilon$-pure blockade in
a graph $G$. 
Let $J$ be the \textit{pattern graph} of $\mathcal B$ with $V(J)=[\ell]$ such that $ij\in E(J)$ if and only if $B_i,B_j \text{ are } \varepsilon\text{-sparse to each other in }\overline G.$
Since $\mathcal B$ is $\varepsilon$-pure, every non-edge of $J$ corresponds
to a pair of blocks that is $\varepsilon$-sparse to each other in $G$.

The pattern graph records only the relations between blocks. 
In this paper, a \textit{homogeneous set} in a graph is a vertex set that is
either a clique or a stable set.
The next
lemma shows that a homogeneous set in this pattern already gives a
restricted induced subgraph in the original graph.

\begin{lemma}\label{lem:pattern-lift}
Let $\varepsilon>0$, and let $\mathcal B=(B_1,\ldots,B_\ell)$ be an
$\varepsilon$-pure blockade in a graph $G$, with
$|B_i|=m$, for all $i\in[\ell]$. Let $J$ be the pattern graph of $\mathcal B$.
If $J$ contains a clique or a stable set $I$ with $|I|=q$, then
$G[\bigcup_{i\in I} B_i]$ is $(\varepsilon+q^{-1})$-restricted. In particular,
$G$ contains an $(\varepsilon+q^{-1})$-restricted induced subgraph on $qm$
vertices.
\end{lemma}

\begin{proof}
Let $I\subseteq[\ell]$ be a clique or stable set in $J$ of size $q$, and let $S=\bigcup_{i\in I}B_i$.  If $I$ is stable in $J$, then for all distinct $i,j\in I$, the pair $B_i,B_j$ is $\eps$-sparse to each other in $G$.  Hence, for every $v\in B_i\subseteq S$,
\[
        d_{G[S]}(v)
        \le (m-1)+(q-1)\eps m
        \le (1/q+\eps)|S|.
\]
Thus $G[S]$ is $(\eps+1/q)$-sparse.  If $I$ is a clique in $J$, the same argument in $\overline G$ shows that $\overline{G[S]}$ is $(\eps+1/q)$-sparse.  
Therefore, $G[S]$ is $(\eps+1/q)$-restricted.
\end{proof}

\subsection{Universal bigraphs and external VC-dimension}

A \emph{bigraph} is a graph together with a specified bipartition $(V_1,V_2)$.
If $H$ is a bigraph, a \emph{bi-induced copy} of $H$ in a graph $G$ is an
injective map $\phi:V(H)\to V(G)$ such that, for all $x\in V_1(H)$ and
$y\in V_2(H)$, we have $xy\in E(H)$ if and only if
$\phi(x)\phi(y)\in E(G)$.
No condition is imposed on pairs of vertices that lie on the same side of
the bigraph.

For a graph $G$, a set $S\subseteq V(G)$ is \emph{externally shattered}
if, for every $A\subseteq S$, there exists a vertex
$v_A\in V(G)\setminus S$ such that
$
        N_G(v_A)\cap S=A.
$
The \emph{external VC-dimension} of $G$, denoted $\eVC(G)$, is the maximum
size of an externally shattered set.

For a set $X$, let $\mathcal P(X)$ denote the set of all subsets of $X$.
For an integer $r\ge 0$, let $U_r$ be the bigraph with bipartition
$([r],\mathcal P([r]))$ 
such that for $i\in [r]$ and $A\in\mathcal P([r])$, we have
$iA\in E(U_r)$ if and only if $i\in A$.
Let $\mathcal C_r$ be the class of graphs with no bi-induced copy of
$U_{r+1}$. Then $\mathcal C_r$ is exactly the class of graphs with
external VC-dimension at most $r$. It follows from the definitions that if
$\VC(G)\le d$, then $G\in\mathcal C_d$.
In this paper, $\log$ denotes the logarithm of base $2$.



\begin{lemma}\label{lem:external-to-vc}
For $r\ge 0$,
if $G\in\mathcal C_r$, then
\[
        \VC(G)< r+1+\ceil{\log_2(r+2)}.
\]
\end{lemma}

\begin{proof}
Suppose for a contradiction that there is a shattered set
$S\subseteq V(G)$ with
$$
 m:=|S|\ge r+1+\ceil{\log_2(r+2)}.
$$
Choose $T\subseteq S$ with $|T|=r+1$. We shall show that $G$ contains a
bi-induced copy of $U_{r+1}$.

For a fixed set $A\subseteq T$, there are exactly $2^{m-r-1}$ subsets $B_i,i\in[2^{m-r-1}]$ of $S$ such
that $B_i\cap T=A$, and by the choice of $m$ we have
$
2^{m-r-1}\ge r+2.
$
Since $S$ is shattered, for each such set $B_i$ there is a vertex $v_{B_i}\in V(G)$
such that
$
N_G(v_{B_i})\cap S=B_i.
$
For every different $i,j\in[2^{m-r-1}]$, 
$B_i$ and $B_j$ are distinct; hence the vertices $v_{B_i}$ are
distinct. Since $|T|=r+1$, at most $r+1$ of these vertices belong to $T$.
It follows that at least one vertex lies in $V(G)\setminus T$.
Consequently, for
every $A\subseteq T$, we may choose a vertex $w_A\in V(G)\setminus T$ such
that
$
N_G(w_A)\cap T=A.
$
Moreover, the vertices $w_A$ with $A\subseteq T$ are pairwise distinct.
Thus, after identifying
$T$ with $[r+1]$, the map sends one side of $U_{r+1}$ to $T$ and
the vertex on the other side indexed by $A\subseteq [r+1]$ to $w_A$,
and gives a bi-induced
copy of $U_{r+1}$ in $G$.
This contradicts the assumption that
$G\in\mathcal C_r$ and the lemma holds.
\end{proof}

For $\eps>0$, two disjoint vertex sets $A,B$ in a graph $G$ are
\emph{weakly $\eps$-sparse} in $G$ if there are at most
$\eps |A||B|$ edges of $G$ between $A$ and $B$. They are
\emph{weakly $\eps$-pure in $G$} if they are weakly $\eps$-sparse in
$G$ or weakly $\eps$-sparse in $\overline G$.
We shall use the following quantitative form of the ultra-strong regularity
lemma. It is due to Lovász and Szegedy~\cite[Section~4]{LovaszSzegedy2010},
with the polynomial dependence supplied by Fox, Pach and Suk~\cite[Theorem~1.3]{FoxPachSuk2019}.
We use it in the form stated by Nguyen, Scott and Seymour~\cite[Theorem~2.1]{NguyenScottSeymour2025}.
An \emph{equipartition} of a set $S$ is a partition
$S_1\cup\cdots\cup S_k$ of $S$ such that $|S_i|$ and $|S_j|$ differ by at
most one for all $i,j\in[k]$.

\begin{theorem}[Nguyen, Scott and Seymour \cite{NguyenScottSeymour2025}]\label{thm:usr}
There exists $K\ge 1$ such that the following holds. For
every $d\ge 1$, every $\eps\in(0,1/2)$, and every graph $G$ with
$\VC(G)\le d$, there is an equipartition
$V(G)=V_1\cup\cdots\cup V_L$, where
$\eps^{-1}\le L\le \eps^{-Kd}$, such that all but at most 
$\eps$-fraction of the pairs $(V_i,V_j)$ are weakly $\eps$-pure.
\end{theorem}

We find a pure blockade in graphs with external
VC-dimension at most $r$.

\begin{lemma}\label{lem:pure-blockade}
There exists $b\ge 1$ such that the following holds.
Let $r\ge 1$, $\eps\in(0,1/2)$ and $G\in\mathcal C_r$ with
$|G|\ge \eps^{-br}$. Then $G$ contains an
$(\eps^{-1},\eps^{br}|G|)$-blockade
$\mathcal B=(B_1,\ldots,B_\ell)$ such that $|B_i|=m$ for all
$i\in[\ell]$, where $m\le \eps^2|G|$, and for all distinct
$i,j\in[\ell]$, the pair $B_i,B_j$ is $\eps$-pure.
\end{lemma}

\begin{proof}
Let $K\ge 1$ be the constant from Theorem~\ref{thm:usr}. Choose
$
        b=34K .
$
By Lemma~\ref{lem:external-to-vc}, every graph in $\mathcal C_r$ has
VC-dimension at most $3r$. Let $\delta=\eps^4/100$;
by Theorem~\ref{thm:usr}, we obtain an equipartition
$V(G)=V_1\cup\cdots\cup V_L$, where
$\delta^{-1}\le L\le \delta^{-3Kr}$, such that all but at most 
$\delta$-fraction of the pairs $(V_i,V_j)$ are weakly $\delta$-pure.

Let $R$ be the graph with vertex set $[L]$ in which two distinct indices
$i,j$ are adjacent if and only if the pair $(V_i,V_j)$ is not weakly
$\delta$-pure. Then $e(R)\le \delta L^2$, and by Turán's theorem,
$R$ contains a stable set $J$ with $|J|\ge (3\delta)^{-1}$. Since the
partition is an equipartition, at most two different part sizes occur.
Since $(3\delta)^{-1}\ge 30\eps^{-4}$, we may choose a set
$I\subseteq J$ with
$
        |I|=\ell=\ceil{\eps^{-1}}
$
such that $|V_i|=|V_j|$ for all $i,j\in I$.

Let $M=|V_i|$ for all $i\in I$. For each $i\in I$, we define
$B_i\subseteq V_i$ as follows. For every $j\in I\setminus\{i\}$, if
$(V_i,V_j)$ is weakly $\delta$-sparse in $G$, delete every vertex
$v\in V_i$ when $v$ has more than $\eps M/2$ neighbours in $V_j$; and if
$(V_i,V_j)$ is weakly $\delta$-sparse in $\overline G$, delete every vertex
$v\in V_i$ when $v$ has more than $\eps M/2$ non-neighbours in $V_j$.

For a fixed $j\in I\setminus\{i\}$, a simple counting argument shows that
we delete at most $2\delta M/\eps$ vertices from $V_i$.
Hence by the choice of $\delta$, the
total number of vertices deleted from $V_i$ is at most
\[
        \ell\cdot \frac{2\delta}{\eps}M
        \le 3\eps^{-1}\frac{2\delta}{\eps}M
        < \frac M2 .
\]
Thus $|B_i|\ge M/2$ for every $i\in I$. 
We may assume that $|B_i|=m$ with $m\ge M/2$ for all $i\in I$ by removing extra vertices if necessary.

\begin{claim}\label{claim:pair}
For all distinct $i,j\in I$, the pair $(B_i,B_j)$ is $\eps$-pure.
\end{claim}

{\renewcommand{\qedsymbol}{$\blacksquare$}
\begin{proof}[Proof of the Claim \ref{claim:pair}]
Suppose first that $(V_i,V_j)$ is weakly $\delta$-sparse in $G$. Then every
vertex in $B_i$ has at most $\eps M/2$ neighbours in $V_j$, and hence at most
$\eps M/2$ neighbours in $B_j$. Since $|B_j|\ge M/2$, this is at most
$\eps |B_j|$. The same argument with $i$ and $j$ shows that $B_j$ is
$\eps$-sparse to $B_i$ in $G$. Thus $B_i,B_j$ are $\eps$-sparse to each other
in $G$. The case in which $(V_i,V_j)$ is weakly $\delta$-sparse in
$\overline G$ is similar, and this proves the claim.
\end{proof}
}

It remains only to check the size of the blocks. Since
$
        L\le \delta^{-3Kr}
        =100^{3Kr}\eps^{-12Kr},
$
and since $\eps<1/2$, we have
\[
        \eps^{-br}
        =\eps^{-12Kr}\eps^{-22Kr}
        \ge \eps^{-12Kr}2^{22Kr}
        \ge 4\cdot 100^{3Kr}\eps^{-12Kr}
        =4\delta^{-3Kr}
        \ge 2L .
\]
Thus, by the hypothesis $|G|\ge \eps^{-br}$, we have $|G|\ge 2L$.
Therefore each part of the equipartition has size between $|G|/(2L)$ and
$2|G|/L$. Hence
\[
        |B_i|\le |V_i|\le \frac{2|G|}{L}\le 2\delta |G|
        \le \eps^2|G|.
\]
Also,
\[
        |B_i|\ge \frac M2\ge \frac{|G|}{4L}
        \ge \eps^{br}|G|.
\]
Thus $(B_i:i\in I)$ is an $(\eps^{-1},\eps^{br}|G|)$-blockade whose blocks
have equal size at most $\eps^2|G|$, and every two distinct blocks are
$\eps$-pure. This proves the lemma.
\end{proof}

\subsection{Dimension drop when a vertex is mixed}

Let $X\subseteq V(G)$ and let $v\in V(G)\setminus X$. We say that $v$ is
\emph{mixed on} $X$ if $v$ has both a neighbour and a non-neighbour in $X$.
Otherwise, $v$ is \emph{pure} on $X$.
The following key lemma allows us to induct on the external VC-dimension instead of the number of vertices of a large forbidden bigraph.

\begin{lemma}\label{lem:dimension-drop}
Let $r\ge 1$ and $F\in\mathcal C_r$. 
Assume that
$\mathcal B=(B_1,\ldots,B_\ell)$ is an $\eps$-pure blockade in $F$ and
$J$ is the pattern graph of $\mathcal B$ in $F$. Suppose that
$w\in V(F)\setminus\bigcup_{i\in[\ell]}B_i$ is mixed on $B_i$ for every
$i\in[\ell]$. If $2^{r+1}\eps<1$, then $J\in\mathcal C_{r-1}$.
\end{lemma}

\begin{proof}
Suppose not. We may assume that  $J$ contains a bi-induced copy of $U_r$. 
After relabeling, we may assume that the blocks $B_1,\ldots,B_r$
correspond to one side of this copy of $U_r$. For each
$A\subseteq [r]$, let $C_A$ be the block corresponding to the vertex of another side
indexed by $A$.
 These blocks are all
distinct, and by the definition of the pattern graph, for every
$i\in[r]$ and every $A\subseteq[r]$, the pair $B_i,C_A$ is
$\eps$-sparse to each other in $\overline F$ if $i\in A$, and is
$\eps$-sparse to each other in $F$ if $i\notin A$.

Since $w$ is mixed on $C_A$, for every $A\subseteq[r]$ we may choose two
vertices $u_{A,0},u_{A,1}\in C_A$ such that $u_{A,0}$ is nonadjacent to
$w$, and $u_{A,1}$ is adjacent to $w$.

We claim that for each $i\in[r]$, there exists a vertex $x_i\in B_i$ such that
for every $A\subseteq[r]$ and every $b\in\{0,1\}$ the vertex $x_i$ is
adjacent to $u_{A,b}$ if and only if $i\in A$. Fix $i\in[r]$. If $i\in A$,
then $B_i,C_A$ are $\eps$-sparse to each other in $\overline F$, and so
$u_{A,b}$ has at most $\eps |B_i|$ non-neighbours in $B_i$. If $i\notin A$,
then $B_i,C_A$ are $\eps$-sparse to each other in $F$, and so $u_{A,b}$ has
at most $\eps |B_i|$ neighbours in $B_i$. 
Thus, for each pair $(A,b)$, at most $\eps |B_i|$ vertices of $B_i$ fail
the required adjacency condition with $u_{A,b}$. Since there are $2^{r+1}$ choices of $(A,b)$ and $2^{r+1}\eps<1$, fewer
than $|B_i|$ vertices of $B_i$ fail at least one of these conditions. Hence there exists a vertex
$x_i\in B_i$ satisfying all the required conditions. This proves the claim.

Now consider the vertices $w,x_1,\ldots,x_r$ and the vertices
$u_{A,b}$, where $A\subseteq[r]$ and $b\in\{0,1\}$. These vertices are all
distinct,
since the blocks $B_1,\ldots,B_r$ and $C_A$, $A\subseteq[r]$, are
distinct, while $w$ lies outside all blocks of the blockade, and
$u_{A,0}\ne u_{A,1}$ by their adjacency to $w$.

We identify one side of $U_{r+1}$ with $\{0\}\cup[r]$, where $0$
corresponds to $w$ and $i$ corresponds to $x_i$. For each
$A\subseteq[r]$ and $b\in\{0,1\}$, let $A_b=A$ if $b=0$ and $A_b=A\cup\{0\}$ if $b=1$. Then
$u_{A,b}$ is adjacent to $w$ if and only if $0\in A_b$, and, by the claim,
$u_{A,b}$ is adjacent to $x_i$ if and only if $i\in A_b$. Thus these
vertices form a bi-induced copy of $U_{r+1}$ in $F$, contrary to
$F\in\mathcal C_r$.
This proves the lemma.
\end{proof}

\subsection{Restricted induced subgraph or complete or anticomplete blockade}

Next we handle the sparse case.  Starting from a large pure blockade, either
many mixed incidences allow us to apply the dimension-drop lemma and lift
a homogeneous pattern, or one block has a large anticomplete set
outside it.
Throughout the next few lemmas, we use the induction hypothesis in the
following form. Fix $a\ge \max\{r,2\}$ such that
$\max\{\omega(J),\alpha(J)\}\ge |J|^{1/a}$ for every $J\in\mathcal C_{r-1}$.

\begin{lemma}\label{lem:sparse-pair}
Let $b\ge 1$ be given by Lemma~\ref{lem:pure-blockade}, $r\ge 1$ and $a\ge \max\{r,2\}$.
Assume that every
$n$-vertex graph in $\mathcal C_{r-1}$ contains a clique or a stable set of
size at least $n^{1/a}$.
There exist constants $c\in(0,1/16)$ and $h_1\ge 1$ such that
if
$F\in\mathcal C_r$ is $y$-sparse, where $0<y<c$,
then 
one of the following holds:
\begin{enumerate}[(i)]
\item $F$ contains a $y^9$-restricted induced subgraph with at least
$y^{h_1ra}|F|$ vertices;
\item there is an anticomplete pair $(A,B)$ in $F$ such that
$|A|\ge y^{h_1ra}|F|$ and $|B|\ge (1-3y)|F|$.
\end{enumerate}
\end{lemma}

\begin{proof}
Choose $c>0$ sufficiently small with $0<y<c$ such that
 $2^{r+1}\eps<1$ with $\eps=y^{12a}$. Increase $h_1$ if necessary such that
$h_1\ge 12b$.
We may assume that $|F|>\eps^{-br}$; for otherwise
(i) holds as $y^{h_1ra} |F| < 1$. 
By Lemma~\ref{lem:pure-blockade} and removing blocks if necessary,
there is an $\eps$-pure blockade
$\mathcal B=(B_1,\ldots,B_\ell)$ in $F$ with
$\ell=\ceil{\eps^{-1}}$, with $|B_i|=m$ for all $i\in[\ell]$, and with
$m\ge \eps^{br}|F|$. In particular,
$
m\ge y^{12bra}|F|
$
and, since $m\le \eps^2|F|$,
$$
\ell m\le 2\eps^{-1}\eps^2|F|=2\eps |F|\le y|F|.
$$
Let
$
D=V(F)\setminus \bigcup_{i\in[\ell]}B_i.
$
Suppose that some vertex $w\in D$ is mixed on at least $y\ell$ of the
blocks. Let $\mathcal B'$ be the subblockade consisting of these blocks,
and let $J$ be the pattern graph of $\mathcal B'$ in $F$. Since
$2^{r+1}\eps<1$, Lemma~\ref{lem:dimension-drop} gives
$J\in\mathcal C_{r-1}$. Also,
$
|J|\ge y\ell\ge y\eps^{-1}=y^{1-12a}.
$
By the induction hypothesis, $J$ contains a clique or a stable set $I$ with
$
|I|\ge |J|^{1/a}\ge y^{-12+1/a}\ge y^{-11}.
$
By applying Lemma~\ref{lem:pattern-lift} to the corresponding subblockade, we have that
$
S=\bigcup_{i\in I}B_i
$
induces an $(\eps+|I|^{-1})$-restricted subgraph of $F$. By the choice of
$c$,
$
\eps+|I|^{-1}\le \eps+y^{11}\le y^9.
$
Thus $F[S]$ is $y^9$-restricted. Moreover,
$$
|S|\ge m\ge y^{12bra}|F|\ge y^{h_1ra}|F|,
$$
and so (i) holds.

We may therefore assume that every vertex of $D$ is mixed on fewer than
$y\ell$ blocks. Thus, there exists $i\in[\ell]$ such that at most $y|D|$ vertices of $D$
are mixed on $B_i$.
Let $X$ be the set of vertices in $D$ that are complete to $B_i$. Since
$F$ is $y$-sparse, we have $|X|\le y|F|$; otherwise every vertex of $B_i$
would have more than $y|F|$ neighbours in $F$, a contradiction.
Let $B$ be the set of vertices in $D$ that are anticomplete to $B_i$.
Every vertex of $D$ that is neither mixed on nor complete to $B_i$ belongs
to $B$. Hence
$$
|B|\ge |D|-y|D|-y|F|
     = |F|-\ell m-y|D|-y|F|
     \ge (1-3y)|F|,
$$
using $|D|\le |F|$ and $\ell m\le y|F|$.  
We obtain an
anticomplete pair $(B_i,B)$ in $F$. Finally,
$$
|B_i|=m\ge y^{12bra}|F|\ge y^{h_1ra}|F|.
$$
Thus (ii) holds.
This proves the lemma.
\end{proof}

The previous lemma gives only one large anticomplete pair.  We now repeat
the same argument inside the large remainder.  This either produces a more
restricted induced subgraph, or grows the pair into a complete or
anticomplete blockade.

\begin{lemma}\label{lem:pair-to-blockade}
Let $c$ be the constant given by Lemma~\ref{lem:sparse-pair}, $r\ge 1$, and $a\ge \max\{r,2\}$.
Assume that every $n$-vertex
graph in $\mathcal C_{r-1}$ contains a clique or a stable set of size at
least $n^{1/a}$.
There exists $h_2\ge 1$ such that 
if $F\in\mathcal C_r$ is $y$-restricted where
$0<y<c^2$, then one of the following holds:
\begin{enumerate}[(i)]
\item $F$ contains a $y^4$-restricted induced subgraph with at least
$y^{h_2ra}|F|$ vertices;
\item there is a complete or anticomplete blockade in $F$ of length at
least $y^{-1/2}$ and width at least $y^{h_2ra}|F|$.
\end{enumerate}
\end{lemma}

\begin{proof}
Let $z=y^{1/2}$.  
Since $F$ is $y$-restricted, by replacing $F$ with $\overline F$, we may assume that $F$ is $y$-sparse.
The graph $F$ is also $z$-sparse.

Starting with
$R_0=V(F)$, we construct sets $R_j$ inductively. Suppose that $R_j$ has
already been constructed until outcome (i) occurs. We apply
Lemma~\ref{lem:sparse-pair} to the induced graph $F[R_j]$, with $z$ in place
of $y$ in Lemma~\ref{lem:sparse-pair}.
Then we check the hypothesis of Lemma~\ref{lem:sparse-pair}. First,
$F[R_j]\in\mathcal C_r$, since $\mathcal C_r$ is hereditary. Secondly,
$z<c$, as $y<c^2$. Thirdly, $F[R_j]$ is $z$-sparse. Indeed, by construction
we have
$|R_j|\ge (1-3z)^j|F|.$
Since $j<z^{-1}$, after decreasing $c$ if necessary, there is an absolute
constant $c'>0$ such that
$
        |R_j|\ge c'|F|.
$
We also assume $c\le c'$. Hence, since $z<c$, we have $z\le c'$, and so
\[
        \Delta(F[R_j])
        \le \Delta(F)
        \le y|F|
        = z^2|F|
        \le z|R_j|.
\]
Thus $F[R_j]$ is $z$-sparse.
Let $h_1$ be the constant from Lemma~\ref{lem:sparse-pair}, then by Lemma~\ref{lem:sparse-pair}, one of the following two outcomes holds:
\begin{enumerate}[(i)]
\item there is a $z^9$-restricted induced subgraph $S\subseteq R_j$
such that
$
        |S|\ge z^{h_1ra}|R_j|;
$

\item there is an anticomplete pair $(A_j,R_{j+1})$ in $F[R_j]$
such that
$
        |A_j|\ge z^{h_1ra}|R_j|
$ and $
        |R_{j+1}|\ge (1-3z)|R_j|.
$
\end{enumerate}

If Lemma~\ref{lem:sparse-pair} (i) holds, then $S$ is $y^4$-restricted, since
$z^9=y^{9/2}\le y^4$. Choose $h_2$ large enough so that, for all $0<y<c^2$, $r\ge1$ and $a\ge2$,
$
        c' y^{h_1ra/2}\ge y^{h_2ra}.
$
Then 
$$
|S|\ge z^{h_1ra}|R_j|
    \ge c'z^{h_1ra}|F|
    \ge y^{h_2ra}|F|,
$$
and the 
outcome (i)  holds.

If Lemma~\ref{lem:sparse-pair} (ii) holds, set $R_{j+1}$ to be the second set in the
anticomplete pair $(A_j,R_{j+1})$, and continue inside $R_{j+1}$.
Each extracted block $A_j$ has size at least
$
        c' z^{h_1ra}|F|\ge y^{h_2ra}|F|
$
for a suitable constant $h_2$.  Moreover, for $i<j$, we have
$A_j\subseteq R_{i+1}$, and $A_i$ is anticomplete to $R_{i+1}$. Hence
$(A_0,A_1,\dots,A_{\ceil{y^{-1/2}}-1})$ is an anticomplete blockade in $F$ or $\overline{F}$. This gives a
complete or anticomplete blockade in $F$ of length at least $y^{-1/2}$.
So the outcome (ii) holds. 
This proves the lemma.
\end{proof}

\subsection{Iterative sparsification}

We derive a variant of the minimal-parameter lemma used in the iterative sparsification method due to Nguyen, Scott and Seymour \cite[Lemma 3.2]{NguyenScottSeymour2025}.  We first derive the initial step for the global alternative.

\begin{lemma}\label{lem:constant-start}
Let $c$ be the constant given by Lemma~\ref{lem:sparse-pair}, and let
$b$ be the constant given by Lemma~\ref{lem:pure-blockade}. There are
constants $x\in(0,c^2)$ and $d\ge 1$ such that the following
holds. For every $r\ge 1$ and every graph $G\in\mathcal C_r$, the
graph $G$ contains an $x$-restricted induced subgraph with at least
$x^{dr}|G|$ vertices.
\end{lemma}

\begin{proof}
Choose a sufficiently small constant $x\in(0,c^2)$ and let
$
        q=\ceil{4/x}.
$
Choose a constant $\tau\in(0,1/2)$ such that
\[
        \tau+1/q\le x
        \qquad\text{and}\qquad
        \tau^{-1}\ge R(q,q),
\]
where $R(q,q)$ is the ordinary two-color Ramsey number.

Let $G\in\mathcal C_r$; we may assume that $|G|\ge\tau^{-br}$; for otherwise the result holds.  By Lemma~\ref{lem:pure-blockade} applied with $\eps=\tau$, the graph $G$ contains a $\tau$-pure blockade of length at least $\tau^{-1}$ and width at least $\tau^{br}|G|$.  
Let $J$ be its pattern graph.
Since the pattern graph has at least $R(q,q)$ vertices, it contains a clique or stable set of size $q$.  
By Lemma~\ref{lem:pattern-lift}, we obtain a $(\tau+q^{-1})$-restricted
induced subgraph. Since $\tau+q^{-1}\le x$, 
we can   
enlarge $d$ if necessary and we have
$
        \tau^{br}|G|\ge x^{d r}|G|.
$
This proves the lemma.
\end{proof}

We shall use the following lemma as an iteration lemma.  Suppose that,
for each $y$, every $y$-restricted induced subgraph contains a
$y^p$-restricted induced subgraph of size at least a fixed power of
$y$ times its order.  Then this step can be repeated, and hence gives
the desired conclusion for any smaller parameter.

\begin{lemma}\label{lem:iterate-fixed-power}
Let $x\in(0,1)$, $p\ge 2$, $t\ge 1$, and $\gamma\in(0,x)$.
Let $G$ be a graph satisfying:
\begin{enumerate}[(i)]
\item $G$ has an $x$-restricted induced subgraph on at least
$x^t|G|$ vertices; and
\item for every $y\in[\gamma,x]$ and every $y$-restricted induced
subgraph $F$ of $G$ with $|F|\ge y^{2t}|G|$, the graph $F$ contains a
$y^p$-restricted induced subgraph on at least $y^{pt}|F|$ vertices.
\end{enumerate}
Then $G$ has a $\gamma$-restricted induced subgraph on at least
$\gamma^{2pt}|G|$ vertices.
\end{lemma}

\begin{proof}
Choose $y\in[\gamma^p,x]$ minimal such that $G$ has a $y$-restricted
induced subgraph $F$ with
$
        |F|\ge y^{2t}|G|.
$
Such a $y$ exists by (i), since $x^t|G|\ge x^{2t}|G|$.

If $y\ge \gamma$, then by (ii), applied to $F$, the graph $F$ contains a
$y^p$-restricted induced subgraph $F'$ with
\[
        |F'|\ge y^{pt}|F|
        \ge y^{(p+2)t}|G|
        \ge y^{2pt}|G|
        =(y^p)^{2t}|G|.
\]
This contradicts the minimality of $y$, since $y^p<y$ and
$y^p\ge \gamma^p$.

Hence $\gamma^p\le y<\gamma$. But then $F$ is $\gamma$-restricted and
$
        |F|\ge y^{2t}|G|\ge \gamma^{2pt}|G|.
$
This proves the lemma.
\end{proof}

We now combine the preceding lemmas. 
Lemma~\ref{lem:constant-start} then
provides the initial restricted subgraph.
If no large complete or anticomplete
blockade exists, then Lemma~\ref{lem:pair-to-blockade} gives the
$y$-to-$y^4$ improvement needed in Lemma~\ref{lem:iterate-fixed-power}.
Iterating this yields a
restricted induced subgraph of the desired parameter.

\begin{lemma}\label{lem:global-alternative}
Let $c$ be the constant given by Lemma~\ref{lem:sparse-pair},
$x$ and $d$ be the constants given by Lemma~\ref{lem:constant-start}.
Let $r\ge 1$, $a\ge \max\{r,2\}$, and assume that every $n$-vertex
graph in $\mathcal C_{r-1}$ contains a clique or a stable set of size at
least $n^{1/a}$.
There exists $h_3\ge 1$ such that 
for every graph $G\in\mathcal C_r$ and every
$\eta\in(0,x)$, one of the following holds:
\begin{enumerate}[(i)]
\item $G$ contains an $\eta$-restricted induced subgraph with at least
$\eta^{h_3ra}|G|$ vertices;
\item there is a complete or anticomplete blockade in $G$ of length
$k\in[2,\eta^{-1}]$ and width at least $|G|/k^{h_3ra}$.
\end{enumerate}
\end{lemma}

\begin{proof}
Let $h_2$ be the constant given by
Lemma~\ref{lem:pair-to-blockade},
and $u=8(10h_2ra+dr)$. 
Choose $h_3$ large enough such that
$u\le h_3ra$.
We will use Lemma~\ref{lem:iterate-fixed-power} to  prove this lemma.

Suppose that outcome (ii) does not hold. We apply Lemma~\ref{lem:iterate-fixed-power} with
$p=4$, $t=u/8$. 
We first verify the hypothesis of Lemma~\ref{lem:iterate-fixed-power} (ii).
Let $y\in[\eta,x]$ and $F$ be a $y$-restricted induced
subgraph of $G$ such that $|F|\ge y^{u/4}|G|$. 
Since $x<c^2$, we apply Lemma~\ref{lem:pair-to-blockade} to $F$. Hence one
of the following holds:
\begin{enumerate}[(i)]
\item $F$ contains a $y^4$-restricted induced subgraph with at least
$y^{h_2ra}|F|$ vertices;

\item there is a complete or anticomplete blockade in $F$ of length at
least $y^{-1/2}$ and width at least $y^{h_2ra}|F|$.
\end{enumerate}

If Lemma~\ref{lem:pair-to-blockade} (i) holds, then it satisfies the hypothesis of Lemma~\ref{lem:iterate-fixed-power} (ii).
Indeed, since $u/2\ge h_2ra$ and
$0<y<1$, the induced subgraph found in Lemma~\ref{lem:pair-to-blockade} (i) has size at least
$        y^{h_2ra}|F|\ge y^{u/2}|F|.
$
If Lemma~\ref{lem:pair-to-blockade} (ii) holds,
choose a subblockade of length $k=\ceil{y^{-1/2}}$. Since
$y\le x<c^2<1/16$, we have $k\ge 2$; and since $y\ge\eta$, we have
$
k\le 2y^{-1/2}\le 2\eta^{-1/2}\le \eta^{-1}.
$
The width of this blockade is at least
$
y^{h_2ra}|F|\ge y^{h_2ra+u/4}|G|.
$
Since $u/2\ge h_2ra+u/4$ and $0<y<1$, this is at least $y^{u/2}|G|$.
Moreover, $k\ge y^{-1/2}$, and so
$
|G|/k^u\le y^{u/2}|G|.
$
Thus the blockade has width at least $|G|/k^u$, and hence at least
$|G|/k^{h_3ra}$. This gives outcome (ii), a contradiction.

It remains to verify the hypothesis of Lemma~\ref{lem:iterate-fixed-power} (i).
By
Lemma~\ref{lem:constant-start}, the graph $G$ contains an $x$-restricted
induced subgraph with at least $x^{dr}|G|$ vertices. Since $u/8\ge dr$, this
size is at least $x^{u/8}|G|$. 
By applying Lemma~\ref{lem:iterate-fixed-power} with $p=4$ and $t=u/8$, there
exists an $\eta$-restricted induced subgraph with at least
$
\eta^u|G|\ge \eta^{h_3ra}|G|
$
vertices, and outcome (i) holds.
This proves the lemma.
\end{proof}

\subsection{Large homogeneous set in restricted induced subgraph or complete or anticomplete blockade}

It remains to derive an Erd\H{o}s--Hajnal bound.
The next lemma is the standard cograph induction: A restricted subgraph
gives a large homogeneous set directly, while a complete or anticomplete
blockade allows cographs found inside the blocks to be joined together.

\begin{lemma}\label{lem:linear-conversion}
Let $\mathcal C$ be a hereditary graph class. Suppose there exist constants
$\gamma\in(0,1/16)$ and $u\ge 1$ such that, for every $G\in\mathcal C$ and
every $\eta\in(0,\gamma)$, one of the following holds:
\begin{enumerate}[(i)]
\item $G$ has an $\eta$-restricted induced subgraph on at least
$\eta^u|G|$ vertices;
\item $G$ has a complete or anticomplete blockade of length
$k\in[2,\eta^{-1}]$ and width at least $|G|/k^u$.
\end{enumerate}
Then every $n$-vertex graph in $\mathcal C$ contains a clique or stable set
of size at least
$$
n^{1/(h_4u)},
$$
where $h_4$ is a constant depending only on $\gamma$.
\end{lemma}

\begin{proof}
Let $\ell=\log_2(1/\gamma)$; recall
that a cograph is a graph with no induced $P_4$, and that the complete or
anticomplete join of cographs is again a cograph.

\begin{claim}\label{claim:large-cograph}
Every graph $G\in\mathcal C$ with $|G|=n$ contains an induced cograph on at
least $n^{1/(8\ell u)}$ vertices.
\end{claim}
{\renewcommand{\qedsymbol}{$\blacksquare$}
\begin{proof}[Proof of the Claim \ref{claim:large-cograph}]
We prove the claim by induction on $n$.
If $n\le \gamma^{-2u}=2^{2\ell u}$, then the result holds.
Thus we may assume
$n>\gamma^{-2u}$ and set
$
\eta=n^{-1/(2u)}.
$
Then $\eta<\gamma$, so the assumption of the lemma applies.

If (i) holds, let $S$ be an $\eta$-restricted induced subgraph with
$
|S|\ge \eta^u n=n^{1/2}.
$
In $G[S]$ or in $\overline{G[S]}$, the maximum degree is at most
$\eta |S|$. By Turán's theorem, $G[S]$ has a clique or stable set of size
at least
$$
\frac{|S|}{\eta |S|+1}\ge \frac{1}{2\eta}
       =\frac12 n^{1/(2u)}.
$$
Since $n>2^{2\ell u}$ and $\ell\ge 4$, this is at least
$n^{1/(8\ell u)}$. A clique or stable set is a cograph, so the induction
step is complete in this case.

If (ii) holds, let $(B_1,\ldots,B_k)$ be the complete or
anticomplete blockade. By induction applied inside each block, $G[B_i]$
contains an induced cograph of size at least $|B_i|^{1/(8\ell u)}$. Combining
these cographs across the complete or anticomplete blockade gives an
induced cograph of size at least
$$
k\left(\frac n{k^u}\right)^{1/(8\ell u)}
        =k^{1-1/(8\ell)}n^{1/(8\ell u)}
        \ge n^{1/(8\ell u)},
$$
since $1/(8\ell)<1$ and $k\ge 2$.
This proves the claim.
\end{proof}
}
By the claim, every $G\in\mathcal C$ contains an induced cograph $Q$ with
$
        |Q|\ge |G|^{1/(8\ell u)}.
$
Every $n$-vertex cograph has a clique or
stable set of size at least $n^{1/2}$. Therefore every $G\in\mathcal C$
satisfies
$$
\max\{\omega(G),\alpha(G)\}\ge |G|^{1/(16\ell u)}.
$$
This proves the lemma with $h_4=16\ell$.
\end{proof}

\section{Proof of the main theorem}\label{sec:main}

We are now ready to prove the main theorem by induction on $r$.  For graphs in $\mathcal C_r$, Lemma~\ref{lem:global-alternative} gives two alternatives of their structures: either a large restricted induced subgraph, or a large complete or anticomplete blockade.
Together with the preceding conversion lemma, this gives a reciprocal Erd\H{o}s--Hajnal exponent for $\mathcal C_r$ in terms of the corresponding reciprocal exponent for $\mathcal C_{r-1}$.

\begin{theorem}\label{thm:Cr}
There exists $h\ge 1$ such that, for every integer
$r\ge 1$ and every graph $G\in\mathcal C_r$,
$$
\max\{\omega(G),\alpha(G)\}\ge |G|^{(hr)^{-r}}.
$$
\end{theorem}

\begin{proof}
Let $c$ be the constant from Lemma~\ref{lem:sparse-pair},
and let $x\in(0,c^2)$ and $d\ge 1$ be the constants given by
Lemma~\ref{lem:constant-start}. Let $h_3$ be the constant given by
Lemma~\ref{lem:global-alternative} and let $h_4$ be the constant given by
Lemma~\ref{lem:linear-conversion}, applied with $\gamma=x$.

For each $r\ge 0$, choose $s_r\ge 1$ such that every graph $G\in\mathcal C_r$ satisfies $$ \max\{\omega(G),\alpha(G)\}\ge |G|^{1/s_r}. $$ We take $s_0=1$; indeed, if a graph has no bi-induced copy of $U_1$, then no vertex has both a neighbour and a non-neighbour, and the graph is
complete or empty.

Assume $r\ge 1$ and that $s_{r-1}$ has been defined. Let
$a=\max\{s_{r-1},r,2\}$ and $u=h_3ra$.
Then by
Lemma~\ref{lem:global-alternative}, for every $\eta\in(0,x)$ and every
graph $G\in\mathcal C_r$, one of the following holds:
\begin{enumerate}[(i)]
\item $G$ contains an $\eta$-restricted induced subgraph with at least
$\eta^u|G|$ vertices;

\item $G$ contains a complete or anticomplete blockade of length
$k\in[2,\eta^{-1}]$ and width at least $|G|/k^u$.
\end{enumerate}

By applying Lemma~\ref{lem:linear-conversion} with $\gamma=x$ to the hereditary
class $\mathcal C_r$, we have
$$
s_r\le h_4u\le h_3h_4r\max\{s_{r-1},r,2\}.
$$

Replacing $s_{r-1}$
by a larger value if necessary,
we may assume that $s_{r-1}\ge r$, for
all $r\ge 2$.
Then
$
s_r\le h_3h_4rs_{r-1}.
$
Since $s_0=1$, there exists a constant $h$ such that
$
s_r\le (h_3h_4)^r r!\le (hr)^r.
$
 Therefore every $G\in\mathcal C_r$
satisfies
$$
\max\{\omega(G),\alpha(G)\}\ge |G|^{1/s_r}\ge |G|^{(hr)^{-r}}.
$$
This proves the theorem.
\end{proof}

It remains to translate the result for the classes $\mathcal C_r$ back to
graphs of bounded VC-dimension.  This is immediate from the definition of
the universal bigraphs.

\begin{proof}[Proof of Theorem~\ref{thm:main}]
If $\VC(G)\le d$, then $G$ cannot contain a bi-induced copy of $U_{d+1}$, since one side of such a copy would be shattered by the other. Hence $G\in\mathcal C_d$. 
By Theorem~\ref{thm:Cr} with $r=d$,
\[
        \max\{\omega(G),\alpha(G)\}\ge |G|^{(hd)^{-d}}.
\]
This proves Theorem~\ref{thm:main}.
\end{proof}

\section{Applications}\label{sec:applications}

We collect several consequences of Theorem~\ref{thm:main}. Throughout this section, let $h$ denote a positive constant that may change from one statement or proof to another, and which is not necessarily the same as the constant in Theorem~\ref{thm:main}. Unless a dependence on fixed parameters is explicitly indicated, $h$ is absolute; we may increase $h$ without further comment. For $q\ge 1$, we write $\mathrm{twr}_1(x)=x$ and $\mathrm{twr}_{q+1}(x)=2^{\mathrm{twr}_q(x)}$. We also write $\log^{(0)} n=n$ and $\log^{(q+1)} n=\log(\log^{(q)} n)$ for the iterated logarithms.

For the hypergraph Ramsey applications, we use the following VC-dimension convention. If $H$ is a $k$-uniform hypergraph on a vertex set $V$ and $A\in\binom{V}{k-1}$, let
$
        N_H(A)=\{v\in V\setminus A:A\cup\{v\}\in E(H)\}.
$
We define $\VC(H)$ to be the VC-dimension of the family $\{N_H(A):A\in\binom{V}{k-1}\}.$ 
For $k=2$, this is exactly the usual VC-dimension of the neighbourhood family of a graph. If $\chi:\binom{[N]}{k}\to[r]$ is an $r$-coloring, we write $H_i$ for the $k$-uniform hypergraph formed by the edges of color $i$. We call such a coloring \emph{admissible} if $\VC(H_i)\le d$ for $i=1,\ldots,r-1$; no VC-dimension assumption is imposed on the last color.


For $k\ge2$, $r\ge2$ and $d\ge1$, let $R_{k,r}^d(t)$ be the smallest integer $N$ such that every admissible $r$-coloring of $\binom{[N]}{k}$ contains a monochromatic set of size $t$. Equivalently, let $f_{k,r}^d(n)$ be the largest integer $t$ such that every admissible $r$-coloring of $\binom{[n]}{k}$ contains a monochromatic set of size $t$. In the two-color case we abbreviate $R_k^d(t)=R_{k,2}^d(t)$ and $f_k^d(n)=f_{k,2}^d(n)$. The Ramsey estimates below use the classical Erd\H{o}s--Rado stepping-up argument \cite{ErdosRado1952}, while the bounded VC-dimension applications and polynomial R\"odl consequences follow the reductions of Nguyen, Scott and Seymour \cite[Sections~4--6]{NguyenScottSeymour2025}. 
We shall use the following two standard estimates.
\begin{lemma}[Sauer--Shelah~\cite{VapnikChervonenkis1971,Sauer1972,Shelah1972}]\label{lem:sauer-shelah}
Let $\mathcal F$ be a family of subsets of a finite set $X$, and suppose that $\VC(\mathcal F)\le d$. Then, for every finite set $S\subseteq X$,
$
        |\{F\cap S:F\in\mathcal F\}| \le \sum_{j=0}^{d}\binom{|S|}{j}.
$
In particular,
\[
        |\{F\cap S:F\in\mathcal F\}|
        \le
        (|S|+1)^{\max\{1,d\}}.
\]
\end{lemma}
\begin{lemma}[Warren-type sign-pattern bound~\cite{Warren1968}]\label{lem:sign-pattern}
There is an absolute constant $h$ such that the following holds. Let
$P_1,\ldots,P_s$ be real polynomials in $m$ variables, each of degree at most
$\Delta$. Then the number of weak sign patterns of the form
$
        (\mathbf 1_{P_1(x)\ge 0},\ldots,\mathbf 1_{P_s(x)\ge 0}),
$
as $x$ ranges over $\mathbb R^m$, is at most
$
        (h\Delta s)^m .
$
\end{lemma}

\subsection{Ramsey bounds for hypergraph colorings with bounded VC-dimension}

The first consequence is a multicolor Ramsey bound for hypergraph colorings with bounded VC-dimension.

\begin{corollary}\label{cor:multi-hypergraph-ramsey}
For every fixed $k\ge2$ and $r\ge2$, there is a constant $h_{k,r}$ such that for all integers $d\ge1$ and $t\ge k$,
$$
        R_{k,r}^d(t)
        \le
        \mathrm{twr}_{k-1}\!\left(
        t^{h_{k,r}(hd)^{d(r-1)}}\right).
$$
Equivalently, for $k\ge3$ and all sufficiently large $n$,
$$
        f_{k,r}^d(n)
        \ge
        \left(\log^{(k-2)} n\right)^{
        1/(h_{k,r}(hd)^{d(r-1)})}.
$$
\end{corollary}

\begin{proof}
For $k=2$, set $a_d=(hd)^d$. We prove by induction on $r$ that
$
        R_{2,r}^d(t)\le t^{a_d^{r-1}}.
$
For $r=2$, this is Theorem~\ref{thm:main}. For $r>2$, apply Theorem~\ref{thm:main} to the graph formed by color $1$. A clique gives a monochromatic set in color $1$; a stable set eliminates color $1$, and the induction hypothesis applies to the remaining $r-1$ colors.

For $k\ge3$, the Erd\H{o}s--Rado greedy construction gives
$$
        R_{k,r}^d(t)
        \le
        r^{\binom{p}{k-1}}+k-2,
        \qquad
        p=R_{k-1,r}^d(t-1).
$$
When passing to a link, the VC-dimension of each of the first $r-1$ color classes does not increase. Iterating the recurrence from the graph-level bound gives the stated tower estimate.
\end{proof}

\subsection{Boolean combinations of relations of bounded VC-dimension}

We next give elementary ways to verify the hypothesis of bounded VC-dimension. Combined with Theorem~\ref{thm:main}, each verification gives an explicit Erd\H{o}s--Hajnal exponent.

\begin{lemma}\label{lem:boolean-vc}
Let $R_1,\ldots,R_t\subseteq V\times V$ be binary relations, and write
$
        R_i(v)=\{u\in V:(v,u)\in R_i\}.
$
Assume that the row family $\{R_i(v):v\in V\}$ has VC-dimension at most $d_i$. Let $G$ be a graph on $V$ whose edge relation is a Boolean combination of $R_1,\ldots,R_t$: for distinct $u,v\in V$,
$$
        uv\in E(G)
        \quad\Longleftrightarrow\quad
        \Phi(1_{R_1(v,u)},\ldots,1_{R_t(v,u)})=1
$$
for some Boolean function $\Phi:\{0,1\}^t\to\{0,1\}$. We assume that this rule defines a symmetric relation on distinct pairs. Let 
$
        D=1+\sum_{i=1}^t\max\{1,d_i\}.
$
Then
$
        \VC(G)\le hD\log(2D).
$
Consequently, we have 
$$
        \max\{\omega(G),\alpha(G)\}\ge
        |G|^{(hD\log(2D))^{-hD\log(2D)}}.
$$
\end{lemma}

\begin{proof}
Let $S\subseteq V$ have size $s$. By Lemma~\ref{lem:sauer-shelah}, we have
$$
        |\{R_i(v)\cap S:v\in V\}|
        \le
        \sum_{j=0}^{d_i}\binom{s}{j}
        \le
        (s+1)^{\max\{1,d_i\}}.
$$
For fixed $v$, the trace $N_G(v)\cap S$ is determined by the traces $R_i(v)\cap S$, $i=1,\ldots,t$, apart from the diagonal convention that $v$ is not adjacent to itself when $v\in S$. This diagonal issue costs at most a factor of $s+1$. Hence the number of neighbourhood traces on $S$ is at most
$$
        (s+1)\prod_{i=1}^t (s+1)^{\max\{1,d_i\}}
        =
        (s+1)^D.
$$
If $S$ is shattered, then
$
        2^s\le (s+1)^D.
$
This implies
$
        s\le hD\log(2D).
$
Therefore $\VC(G)\le hD\log(2D)$, and the homogeneous-set bound follows from Theorem~\ref{thm:main}.
\end{proof}

\begin{corollary}\label{cor:nip-boolean}
Suppose $G$ is defined by a Boolean combination of formulas
$
        \varphi_1(x;y),\ldots,\varphi_t(x;y),
$
whose corresponding row families have VC-dimensions at most $d_1,\ldots,d_t$. With
$
        D=1+\sum_{i=1}^t\max\{1,d_i\},
$
one has
$$
        \max\{\omega(G),\alpha(G)\}\ge
        |G|^{(hD\log(2D))^{-hD\log(2D)}}.
$$
\end{corollary}

\begin{proof}
Apply Lemma~\ref{lem:boolean-vc} to the relations
$
        R_i(v)=\{u:\varphi_i(v;u)\}.
$
\end{proof}

\begin{proposition}\label{prop:semialgebraic}
Let $m,t,\Delta\ge1$ be integers, and let $P\subseteq\mathbb R^m$ be
finite. Suppose that a graph $G$ on $P$ is defined by polynomials
$$
        f_1,\ldots,f_t\in
        \mathbb R[X_1,\ldots,X_m,Y_1,\ldots,Y_m]
$$
of degree at most $\Delta$, and a Boolean function $\Phi:\{0,1\}^t\to\{0,1\}$, so that for distinct $p,q\in P$,
$$
        pq\in E(G)
        \quad\Longleftrightarrow\quad
        \Phi(1_{f_1(p,q)\ge0},\ldots,1_{f_t(p,q)\ge0})=1.
$$
Then
$
        \VC(G)\le hm\log(2mt\Delta),
$
and hence
$
        \max\{\omega(G),\alpha(G)\}\ge
        |G|^{(hm\log(2mt\Delta))^{-hm\log(2mt\Delta)}}.
$
\end{proposition}

\begin{proof}
Let $S=\{q_1,\ldots,q_s\}\subseteq P$. As $p$ varies, $N_G(p)\cap S$ is determined, up to the diagonal convention, by the weak signs of the $ts$ polynomials
$
        p\mapsto f_j(p,q_\ell),
$ for
$        j=1,\ldots,t$, and $\ell=1,\ldots,s,
$
in $m$ variables, each of degree at most $\Delta$. By Lemma~\ref{lem:sign-pattern}, the number of possible weak sign patterns is at most
$
        (h\Delta ts)^m.
$
Thus the number of neighbourhood traces on $S$ is at most
$
        (s+1)(h\Delta ts)^m.
$
If $S$ is shattered, then
$
        2^s\le (s+1)(h\Delta ts)^m.
$
This implies
$
        s\le hm\log(2mt\Delta).
$
The conclusion follows from Theorem~\ref{thm:main}.
\end{proof}

We next consider graphs whose adjacency matrices have rank at most $r$ over some field. In this case, a direct linear-algebraic argument gives the required VC-dimension bound.

\begin{proposition}\label{prop:low-rank}
Let $G$ be a graph whose adjacency matrix has rank at most $r\ge1$ over some field. Then
$
        \VC(G)\le r.
$
Consequently,
$
        \max\{\omega(G),\alpha(G)\}\ge |G|^{(hr)^{-r}}.
$
\end{proposition}

\begin{proof}
Let $S=\{v_1,\ldots,v_s\}$ be shattered by the open neighbourhoods of $G$. For each $i\in[s]$, there is a vertex $u_i$ such that
$
        N_G(u_i)\cap S=\{v_i\}.
$
Thus the rows of the adjacency matrix, restricted to the columns indexed by $S$, contain the standard basis vectors $e_1,\ldots,e_s$. The row-space projection onto these $s$ coordinates therefore has dimension at least $s$. Hence the rank of the adjacency matrix is at least $s$. It follows that $\VC(G)\le r$, and the homogeneous-set bound follows from Theorem~\ref{thm:main}.
\end{proof}

Combining the preceding rank bound with Lemma~\ref{lem:boolean-vc} gives the corresponding statement for Boolean combinations of graphs whose adjacency matrices have bounded rank.

\begin{corollary}\label{cor:boolean-low-rank}
Suppose the edge relation of $G$ is a Boolean combination of graphs $G_1,\ldots,G_t$, and the adjacency matrix of $G_i$ has rank at most $r_i$ over some field. Let
$
        R=1+\sum_{i=1}^t\max\{1,r_i\}.
$
Then
$$
        \max\{\omega(G),\alpha(G)\}\ge
        |G|^{(hR\log(2R))^{-hR\log(2R)}}.
$$
\end{corollary}

\begin{proof}
By Proposition~\ref{prop:low-rank}, $\VC(G_i)\le r_i$ for every $i$. Then the corollary follows from Lemma~\ref{lem:boolean-vc}.
\end{proof}

We next consider graphs with sign-rank at most $r$. The proof reduces the neighbourhood traces to traces of homogeneous halfspaces in $\mathbb R^r$, whose shatter function has polynomial growth.

\begin{proposition}\label{prop:sign-rank}
Suppose $G$ has sign-rank at most $r$, that is, there is a real matrix $M=(M_{uv})_{u,v\in V(G)}$ with $\operatorname{rank}_{\mathbb R}M\le r$ such that, for distinct $u,v$,
$
        uv\in E(G)$ if and only if $M_{uv}>0. $
Then
$$
        \VC(G)\le hr\log(r+2),
$$
and hence
$$
        \max\{\omega(G),\alpha(G)\}\ge
        |G|^{(hr\log(r+2))^{-hr\log(r+2)}}.
$$
\end{proposition}

\begin{proof}
Factor $M$ as
$
        M_{uv}=\langle a_u,b_v\rangle
$
with $a_u,b_v\in\mathbb R^r$. For fixed $u$, the neighbourhood of $u$ is, up to the diagonal convention, the trace of the homogeneous halfspace
$
        \{x\in\mathbb R^r:\langle a_u,x\rangle>0\}
$
on the finite set $\{b_v:v\in V(G)\}$. On any $s$-point set, homogeneous halfspaces in $\mathbb R^r$ realize at most $(s+1)^r$ traces. The diagonal convention costs another factor of $s+1$. Hence, if an $s$-set is shattered, then
$
        2^s\le (s+1)^{r+1}.
$
This implies
$
        s\le hr\log(r+2).
$
The result follows from Theorem~\ref{thm:main}.
\end{proof}

We next consider graphs defined by a dot-product threshold representation in $\mathbb R^m$. The proof shows that such a representation gives sign-rank at most $m+1$ after adding one extra coordinate.

\begin{corollary}\label{cor:dot-product}
Let $m\ge1$ be an integer, and let $G$ be a graph on $V$. Suppose there are vectors $x_v\in\mathbb R^m$ and a real threshold $\tau$ such that, for distinct $u,v\in V$,
$
        uv\in E(G)
$ if and only if $       
        \langle x_u,x_v\rangle\ge \tau.
$
Then
$$
        \max\{\omega(G),\alpha(G)\}\ge
        |G|^{(hm\log(m+2))^{-hm\log(m+2)}}.
$$
\end{corollary}

\begin{proof}
If $G$ is complete, then the conclusion is trivial. Otherwise, since $V$ is finite and
$
        uv\notin E(G)
$ if and only if $       
        \langle x_u,x_v\rangle<\tau
$
for distinct $u,v$, we may choose
$
        0<\eta<
        \min\{\tau-\langle x_u,x_v\rangle: u\ne v,\ uv\notin E(G)\}.
$
Then, for all distinct $u,v$,
$
        uv\in E(G)
$ if and only if $        
        \langle x_u,x_v\rangle-\tau+\eta/2>0.
$
Moreover,
$$
        \langle x_u,x_v\rangle-\tau+\eta/2
        =
        \langle (x_u,-\tau+\eta/2),(x_v,1)\rangle .
$$
Thus $G$ has sign-rank at most $m+1$. Proposition~\ref{prop:sign-rank} therefore gives
$$
        \max\{\omega(G),\alpha(G)\}\ge
        |G|^{(hm\log(m+2))^{-hm\log(m+2)}},
$$
after adjusting the absolute constant $h$.
\end{proof}

Finally, we consider graphs defined by Boolean combinations of finitely many dot-product threshold inequalities. The proof follows the same sign-pattern counting argument as above.

\begin{proposition}\label{prop:multiple-bilinear}
Let $t\ge1$ and $r_1,\ldots,r_t\ge1$ be integers.
Suppose the edge relation of $G$ is a Boolean combination of $t$ relations
$
        \langle a_u^{(j)},b_v^{(j)}\rangle\ge \tau_j,
j=1,\ldots,t,
$
where $a_u^{(j)},b_v^{(j)}\in\mathbb R^{r_j}$. Let
$
        R=1+\sum_{j=1}^t r_j.
$
Then
$
        \VC(G)\le hR\log(2tR),
$
and hence
$$
       \max\{\omega(G),\alpha(G)\}\ge
        |G|^{(hR\log(2tR))^{-hR\log(2tR)}}.
$$
\end{proposition}

\begin{proof}
Fix $S=\{v_1,\ldots,v_s\}$. As $u$ varies, $N_G(u)\cap S$ is determined by the weak signs of the $ts$ affine-linear functions
$
        \langle a_u^{(j)},b_{v_\ell}^{(j)}\rangle-\tau_j,
        $ for $ j=1,\ldots,t,$ and $ \ell=1,\ldots,s,$
in the parameter vector
$
        (a_u^{(1)},\ldots,a_u^{(t)})\in
        \mathbb R^{\sum_j r_j}.
$
The arrangement of these $ts$ affine hyperplanes has at most $(hts)^R$ faces, and therefore realizes at most $(hts)^R$ weak sign patterns. The diagonal convention adds a factor of at most $s+1$. Thus the number of neighbourhood traces on $S$ is at most
$
        (s+1)(hts)^R.
$
If $S$ is shattered, then
$
        2^s\le (s+1)(hts)^R.
$
This implies
$
        s\le hR\log(2tR).
$
The conclusion follows from Theorem~\ref{thm:main}.
\end{proof}

\subsection{Other quantitative consequences}

We finish this section by recording some further consequences of the reductions of Nguyen, Scott and Seymour~\cite[Sections~4--6]{NguyenScottSeymour2025}. The hypergraph Ramsey estimate under bounded VC-dimension has already been stated in the stronger multicolor form in Corollary~\ref{cor:multi-hypergraph-ramsey}, so we omit its two-color special case here. For the remaining applications, the proofs are the same as in~\cite{NguyenScottSeymour2025}; the only change is that their bounded VC-dimension Erd\H{o}s--Hajnal exponent is replaced by Theorem~\ref{thm:main}, which gives the explicit exponent $(hd)^{-d}$.

\begin{corollary}\label{cor:known-applications}
There exists $h\ge1$ such that the following hold.

\begin{enumerate}
\item If $\VC(G)\le d$, then for every $\eps\in(0,1/2)$, the graph $G$ contains an $\eps$-restricted induced subgraph on at least
$
        \eps^{(hd)^d}|G|
$
vertices.

\item Let $T$ be an $n$-vertex tournament whose in-neighbourhood family has VC-dimension at most $d$.
For a tournament $T$, let $\operatorname{tr}(T)$ denote the maximum order of a transitive subtournament of $T$.
Then $T$ contains a transitive subtournament of size at least
$
        n^{(h d\log d)^{-h d\log d}} .
$
Consequently, if $Q$ is a fixed two-colorable tournament and $d_Q$ bounds the in-neighbourhood VC-dimension of all $Q$-free tournaments, then every $Q$-free tournament $T$ satisfies
$$
        \operatorname{tr}(T)\ge
        |T|^{(h d_Q\log d_Q)^{-h d_Q\log d_Q}} .
$$

\item More generally, if a hereditary graph class $\mathcal G$ has VC-dimension at most $d_0$, then every $G\in\mathcal G$ satisfies
$$
        \max\{\omega(G),\alpha(G)\}\ge |G|^{(hd_0)^{-d_0}}
$$
and contains an $\eps$-restricted induced subgraph on at least
$
        \eps^{(hd_0)^{d_0}}|G|
$
vertices. In particular, this applies to the NIP, stable, distal, semi-algebraic, algebraic, and bounded-crossing geometric classes discussed in \cite{NguyenScottSeymour2025}, whenever the corresponding VC-dimension bound is known.
\end{enumerate}
\end{corollary}

Indeed, the first item is obtained by running the polynomial R\"odl argument of Nguyen, Scott and Seymour with the explicit Erd\H{o}s--Hajnal exponent $(hd)^{-d}$. The induced-free, tournament, and model-theoretic/geometric consequences then follow from the same reductions used in \cite[Sections~5--6]{NguyenScottSeymour2025}. No new ingredient is needed beyond tracking the displayed exponent.

\section{Remarks and a probabilistic obstruction}\label{sec:remarks}

We first indicate where the quantitative improvement comes from in our proof.

\begin{remark}\normalfont
The proof of Nguyen, Scott and Seymour works for every fixed forbidden bigraph
$H$ and inducts on $|H|$ \cite[Section 3]{NguyenScottSeymour2025}. If one
encodes VC-dimension at most $d$ by forbidding the universal bigraph
$U_{d+1}$, then
$$
        |U_{d+1}|=d+1+2^{d+1},
$$
so their induction works on an object of exponential size. This is one
source of the double-exponential lower bound for the exponent, mentioned
after \cite[Theorem 1.2]{NguyenScottSeymour2025}.

The argument above avoids this by working directly with the nested classes $\mathcal C_r$. Indeed, $U_s$ appears as a bi-induced sub-bigraph of $U_{s+1}$ for every $s\ge1$, and hence
$
        \mathcal C_0\subseteq \mathcal C_1\subseteq\cdots .
$
In the mixed-block step, Lemma~\ref{lem:dimension-drop} asserts that
the pattern graph belongs to $\mathcal C_{r-1}$, so the parameter drops from
$r$ to $r-1$, instead of from an arbitrary bigraph to a bigraph with one fewer vertex.
The second optimization is the fixed-power iteration in Lemma~\ref{lem:iterate-fixed-power}.
Let $c_d$
denote the Erd\H{o}s--Hajnal exponent obtained by Nguyen, Scott and Seymour
for the class of graphs of VC-dimension at most $d$; that is, every
$n$-vertex graph of VC-dimension at most $d$ contains a clique or stable set
of size at least $n^{c_d}$. In their iterative sparsification argument, the
local step improves the restrictedness parameter from $y$ to $y^a$, where
$a$ is the reciprocal exponent from the previous induction step. A careful
inspection of their proof gives a recurrence of the form $c_d\ge c_{d-1}^K$,
for a universal constant $K$, and hence $c_d\ge 2^{-2^{O(d)}}$. In our
argument, writing $s_r=1/c_r$ for the reciprocal exponent of $\mathcal C_r$,
only the size
estimates involve $s_{r-1}$ by the dimension-drop lemma, while the restrictedness parameter improves from $y$ to $y^4$
with a fixed power. 
For graphs in $\mathcal C_r$, Lemma~\ref{lem:global-alternative} gives two
structural alternatives: either a large restricted induced subgraph, or a
large complete or anticomplete blockade.
Consequently, we have
$u\le C_1rs_{r-1}$, and the linear conversion lemma gives
$s_r\le C_2u\le Crs_{r-1}$. Hence $s_r\le C^r r!\le (Cr)^r$, and taking $r=d$ yields
$
        c_d\ge (Cd)^{-d}=2^{-O(d\log d)}.
$
\end{remark}

Finally we present an obstruction showing that the Erd\H{o}s--Hajnal
exponent must depend on the VC-dimension.

\begin{remark}\label{rem:random}\normalfont
A random graph construction gives a stronger obstruction than merely saying that the exponent cannot be bounded below independently of $d$.  In fact, for $d\ge 5$, the best possible exponent is at most about $1/d$.

Set
$$
        \rho_d
        =
        \frac{\binom{d+1}{3}}
        {(d-2)(\binom{d+1}{3}-d-1)}
        =
        \frac{d(d-1)}{(d-2)(d-3)(d+2)}
        =
        \frac1d+O\left(\frac1{d^2}\right).
$$
We show that, for every $\epsilon>0$ and all sufficiently large $n$, there is an $n$-vertex graph $G$ with $\VC(G)\le d$ and $\max\{\omega(G),\alpha(G)\}\le n^{\rho_d+\epsilon}$.

Fix $0<\epsilon<1-\rho_d$, and let $G\sim G(n,p)$, where $p=n^{-\rho_d-\epsilon/2}$.  We first bound the VC-dimension. 
Let $D\subseteq V(G)$ with $|D|=d+1$, and let $\mathcal W$ be the family of all $(d-2)$-subsets of $D$.
Recall that the trace of a vertex
$v\in V(G)$ on $D$ is $N_G(v)\cap D$.
Thus $D$ is shattered if for every subset $W\subseteq D$ there is a vertex
$v\in V(G)$ such that $N(v)\cap D=W$. 
  If $D$ is shattered, then every member of $\mathcal W$ occurs as a trace $N(v)\cap D$.  At most $d+1$ of these traces can be witnessed by vertices of $D$ itself. Hence at least $\binom{d+1}{3}-d-1$ traces must be witnessed by vertices outside $D$. Different traces need different witnesses, and a union bound gives
$$
        \Pr(D\text{ is shattered})
        \le
        C_d n^{\binom{d+1}{3}-d-1}
        p^{(d-2)(\binom{d+1}{3}-d-1)},
$$
where $C_d$ depends only on $d$.  Taking a union bound over all choices of
$D$, we obtain
$$
        \Pr(\VC(G)>d)
        \le
        C_d n^{\binom{d+1}{3}}
        p^{(d-2)(\binom{d+1}{3}-d-1)}
        =
        o(1),
$$
by the definition of $\rho_d$.

It remains to control the largest clique and stable set.  By the standard first-moment estimate, if $C=C(d,\epsilon)$ is large enough and $t=Cp^{-1}\log n$, then
$$
        \Pr(\omega(G)\ge t\text{ or }\alpha(G)\ge t)=o(1).
$$
Since $t=Cn^{\rho_d+\epsilon/2}\log n\le n^{\rho_d+\epsilon}$ for all sufficiently large $n$, with positive probability $G$ satisfies both $\VC(G)\le d$ and $\max\{\omega(G),\alpha(G)\}\le n^{\rho_d+\epsilon}.$

Consequently, for the exponent $\eta_d$ defined above, we have $\eta_d\le \rho_d+\epsilon$ for every $\epsilon>0$.  Letting $\epsilon\to0$ gives
$$
        \eta_d
        \le
        \rho_d
        =
        \frac{d(d-1)}{(d-2)(d-3)(d+2)}
        =
        \frac1d+O\left(\frac1{d^2}\right).
$$
Thus no Erd\H{o}s--Hajnal exponent better than order $1/d$ can hold uniformly for graphs of VC-dimension at most $d$.
\end{remark}

\section*{Acknowledgements}

This work was supported by the National Key R\&D Program of China (No.~2022YFA1006400) and the National Natural Science Foundation of China
(No.~12571376).

\bibliographystyle{abbrvnat}
\bibliography{refs}

@article{NguyenScottSeymour2025,
  author  = {Nguyen, Tung and Scott, Alex and Seymour, Paul},
  title   = {Induced subgraph density. {VI}. Bounded {VC}-dimension},
  journal = {Advances in Mathematics},
  volume  = {482},
  pages   = {110601},
  year    = {2025},
  doi     = {10.1016/j.aim.2025.110601},
  eprint  = {2312.15572},
  archivePrefix = {arXiv},
  primaryClass = {math.CO}
}

@incollection{LovaszSzegedy2010,
  author    = {Lov{\'a}sz, L{\'a}szl{\'o} and Szegedy, Bal{\'a}zs},
  title     = {Regularity partitions and the topology of graphons},
  booktitle = {An Irregular Mind},
  series    = {Bolyai Society Mathematical Studies},
  volume    = {21},
  pages     = {415--446},
  publisher = {J{\'a}nos Bolyai Mathematical Society},
  address   = {Budapest},
  year      = {2010}
}

@article{FoxPachSuk2019,
  author  = {Fox, Jacob and Pach, J{\'a}nos and Suk, Andrew},
  title   = {{E}rd{\H{o}}s--{H}ajnal conjecture for graphs with bounded {VC}-dimension},
  journal = {Discrete \& Computational Geometry},
  volume  = {61},
  number  = {4},
  pages   = {809--829},
  year    = {2019},
  doi     = {10.1007/s00454-019-00074-1}
}

@incollection{ErdosHajnal1977,
  author    = {Erd{\H{o}}s, Paul and Hajnal, Andr{\'a}s},
  title     = {On spanned subgraphs of graphs},
  booktitle = {Contributions to Graph Theory and Its Applications},
  pages     = {80--96},
  publisher = {Tech. Hochschule Ilmenau},
  address   = {Ilmenau},
  year      = {1977}
}

@article{ErdosHajnal1989,
  author  = {Erd{\H{o}}s, Paul and Hajnal, Andr{\'a}s},
  title   = {Ramsey-type theorems},
  journal = {Discrete Applied Mathematics},
  volume  = {25},
  number  = {1--2},
  pages   = {37--52},
  year    = {1989},
  doi     = {10.1016/0166-218X(89)90057-2}
}

@article{VapnikChervonenkis1971,
  author  = {Vapnik, Vladimir N. and Chervonenkis, Alexey Y.},
  title   = {On the uniform convergence of relative frequencies of events to their probabilities},
  journal = {Theory of Probability and its Applications},
  volume  = {16},
  number  = {2},
  pages   = {264--280},
  year    = {1971},
  doi     = {10.1137/1116025}
}

@book{AlonSpencer2016,
  author    = {Alon, Noga and Spencer, Joel H.},
  title     = {The Probabilistic Method},
  edition   = {4},
  publisher = {Wiley},
  year      = {2016},
  doi       = {10.1002/9781119061953}
}

@article{Chudnovsky2014,
  author  = {Chudnovsky, Maria},
  title   = {The {E}rd{\H{o}}s--{H}ajnal conjecture---a survey},
  journal = {Journal of Graph Theory},
  volume  = {75},
  number  = {2},
  pages   = {178--190},
  year    = {2014},
  doi     = {10.1002/jgt.21730}
}

@article{FoxSudakov2008,
  author  = {Fox, Jacob and Sudakov, Benny},
  title   = {Induced {R}amsey-type theorems},
  journal = {Advances in Mathematics},
  volume  = {219},
  number  = {6},
  pages   = {1771--1800},
  year    = {2008},
  doi     = {10.1016/j.aim.2008.07.009}
}

@article{BucicFoxPham2024,
  author        = {Buci{\'c}, Matija and Fox, Jacob and Pham, Huy Tuan},
  title         = {Equivalence between {E}rd{\H{o}}s--{H}ajnal and polynomial {R}{\"o}dl and {N}ikiforov conjectures},
  journal       = {arXiv preprint},
  eprint        = {2403.08303},
  archivePrefix = {arXiv},
  primaryClass  = {math.CO},
  year          = {2024}
}

@article{BucicNguyenScottSeymour2024,
  author  = {Buci{\'c}, Matija and Nguyen, Tung and Scott, Alex and Seymour, Paul},
  title   = {Induced subgraph density. {I}. A loglog step towards {E}rd{\H{o}}s--{H}ajnal},
  journal = {International Mathematics Research Notices},
  volume  = {2024},
  number  = {12},
  pages   = {9991--10004},
  year    = {2024},
  doi     = {10.1093/imrn/rnae096}
}

@article{NguyenScottSeymour2023IV,
  author  = {Nguyen, Tung and Scott, Alex and Seymour, Paul},
  title   = {Induced subgraph density. {IV}. {N}ew graphs with the {E}rd{\H{o}}s--{H}ajnal property},
  journal = {Transactions of the American Mathematical Society},
  year    = {2026},
  note    = {Accepted for publication},
  eprint  = {2307.06455},
  archivePrefix = {arXiv},
  primaryClass = {math.CO}
}

@article{NguyenScottSeymour2026VII,
  author  = {Nguyen, Tung and Scott, Alex and Seymour, Paul},
  title   = {Induced subgraph density. {VII}. {T}he five-vertex path},
  journal = {Proceedings of the London Mathematical Society},
  year    = {2026},
  doi     = {10.1112/plms.70133},
  eprint  = {2312.15333},
  archivePrefix = {arXiv},
  primaryClass = {math.CO}
}

@book{Simon2015,
  author    = {Simon, Pierre},
  title     = {A Guide to {NIP} Theories},
  series    = {Lecture Notes in Logic},
  volume    = {44},
  publisher = {Association for Symbolic Logic and Cambridge Scientific Publishers},
  year      = {2015}
}

@article{ChernikovStarchenkoThomas2021,
  author  = {Chernikov, Artem and Starchenko, Sergei and Thomas, Maryanthe Malliaris},
  title   = {Ramsey growth in some {NIP} structures},
  journal = {Journal of the Institute of Mathematics of Jussieu},
  volume  = {20},
  number  = {1},
  pages   = {1--29},
  year    = {2021},
  doi     = {10.1017/S1474748018000335}
}

@article{ChernikovStarchenko2018Distal,
  author  = {Chernikov, Artem and Starchenko, Sergei},
  title   = {Regularity lemma for distal structures},
  journal = {Journal of the European Mathematical Society},
  volume  = {20},
  number  = {10},
  pages   = {2437--2466},
  year    = {2018},
  doi     = {10.4171/JEMS/816}
}

@article{AlonPachPinchasiRadoicicSharir2005,
  author  = {Alon, Noga and Pach, J{\'a}nos and Pinchasi, Rom and Radoi{\v c}i{\'c}, Radoi and Sharir, Micha},
  title   = {Crossing patterns of semi-algebraic sets},
  journal = {Journal of Combinatorial Theory, Series A},
  volume  = {111},
  number  = {2},
  pages   = {310--326},
  year    = {2005},
  doi     = {10.1016/j.jcta.2004.12.008}
}

@article{Warren1968,
  author  = {Warren, Hugh E.},
  title   = {Lower bounds for approximation by nonlinear manifolds},
  journal = {Transactions of the American Mathematical Society},
  volume  = {133},
  number  = {1},
  pages   = {167--178},
  year    = {1968},
  doi     = {10.2307/1994977}
}

@article{FoxPachToth2011,
  author  = {Fox, Jacob and Pach, J{\'a}nos and T{\'o}th, Csaba D.},
  title   = {Intersection patterns of curves},
  journal = {Journal of the London Mathematical Society},
  volume  = {83},
  number  = {2},
  pages   = {389--406},
  year    = {2011},
  doi     = {10.1112/jlms/jdq082}
}

@article{Tomon2024,
  author  = {Tomon, Istv{\'a}n},
  title   = {String graphs have the {E}rd{\H{o}}s--{H}ajnal property},
  journal = {Journal of the European Mathematical Society},
  volume  = {26},
  number  = {1},
  pages   = {275--287},
  year    = {2024},
  doi     = {10.4171/JEMS/1376}
}

@article{Sauer1972,
  author  = {Sauer, Norbert},
  title   = {On the density of families of sets},
  journal = {Journal of Combinatorial Theory, Series A},
  volume  = {13},
  number  = {1},
  pages   = {145--147},
  year    = {1972},
  doi     = {10.1016/0097-3165(72)90019-2}
}

@article{Shelah1972,
  author  = {Shelah, Saharon},
  title   = {A combinatorial problem; stability and order for models and theories in infinitary languages},
  journal = {Pacific Journal of Mathematics},
  volume  = {41},
  number  = {1},
  pages   = {247--261},
  year    = {1972}
}

@article{GishbolinerShapira2023,
  author  = {Gishboliner, Lior and Shapira, Asaf},
  title   = {On {R}{\"o}dl's theorem for cographs},
  journal = {Electronic Journal of Combinatorics},
  volume  = {30},
  number  = {4},
  pages   = {Paper No. 4.13},
  year    = {2023},
  doi     = {10.37236/11603}
}

@article{ErdosRado1952,
  author  = {Erd{\H{o}}s, Paul and Rado, Richard},
  title   = {Combinatorial theorems on classifications of subsets of a given set},
  journal = {Proceedings of the London Mathematical Society},
  series  = {3},
  volume  = {2},
  pages   = {417--439},
  year    = {1952},
  doi     = {10.1112/plms/s3-2.1.417}
}

@article{ConlonFoxPachSudakovSuk2014,
  author  = {Conlon, David and Fox, Jacob and Pach, J{\'a}nos and Sudakov, Benny and Suk, Andrew},
  title   = {Ramsey-type results for semi-algebraic relations},
  journal = {Transactions of the American Mathematical Society},
  volume  = {366},
  number  = {9},
  pages   = {5043--5065},
  year    = {2014},
  doi     = {10.1090/S0002-9947-2014-06044-6}
}

@article{AlonPachSolymosi2001,
  author  = {Alon, Noga and Pach, J{\'a}nos and Solymosi, J{\'o}zsef},
  title   = {Ramsey-type theorems with forbidden subgraphs},
  journal = {Combinatorica},
  volume  = {21},
  number  = {2},
  pages   = {155--170},
  year    = {2001},
  doi     = {10.1007/s004930100016}
}

@article{AlonFischerNewman2007,
  author  = {Alon, Noga and Fischer, Eldar and Newman, Ilan},
  title   = {Efficient testing of bipartite graphs for forbidden induced subgraphs},
  journal = {SIAM Journal on Computing},
  volume  = {37},
  number  = {3},
  pages   = {959--976},
  year    = {2007},
  doi     = {10.1137/050627915}
}

@article{AlonBaloghBollobasMorris2011,
  author  = {Alon, Noga and Balogh, J{\'o}zsef and Bollob{\'a}s, B{\'e}la and Morris, Robert},
  title   = {The structure of almost all graphs in a hereditary property},
  journal = {Journal of Combinatorial Theory, Series B},
  volume  = {101},
  number  = {2},
  pages   = {85--110},
  year    = {2011},
  doi     = {10.1016/j.jctb.2010.10.001},
  eprint  = {0905.1942},
  archivePrefix = {arXiv},
  primaryClass = {math.CO}
}

@article{FoxPachSuk2016,
  author  = {Fox, Jacob and Pach, J{\'a}nos and Suk, Andrew},
  title   = {A polynomial regularity lemma for semi-algebraic hypergraphs and its applications in geometry and property testing},
  journal = {SIAM Journal on Computing},
  volume  = {45},
  number  = {6},
  pages   = {2199--2223},
  year    = {2016},
  doi     = {10.1137/15M1007355},
  eprint  = {1502.01730},
  archivePrefix = {arXiv},
  primaryClass = {math.CO}
}

@article{FoxPachShefferSukZahl2017,
  author  = {Fox, Jacob and Pach, J{\'a}nos and Sheffer, Adam and Suk, Andrew and Zahl, Joshua},
  title   = {A semi-algebraic version of {Z}arankiewicz's problem},
  journal = {Journal of the European Mathematical Society},
  volume  = {19},
  number  = {6},
  pages   = {1785--1810},
  year    = {2017},
  doi     = {10.4171/JEMS/705},
  eprint  = {1407.5705},
  archivePrefix = {arXiv},
  primaryClass = {math.CO}
}

@article{ChernikovGalvinStarchenko2020,
  author  = {Chernikov, Artem and Galvin, David and Starchenko, Sergei},
  title   = {Cutting lemma and {Z}arankiewicz's problem in distal structures},
  journal = {Selecta Mathematica},
  volume  = {26},
  pages   = {25},
  year    = {2020},
  doi     = {10.1007/s00029-020-0551-2},
  eprint  = {1612.00908},
  archivePrefix = {arXiv},
  primaryClass = {math.LO}
}

@article{Do2018,
  author  = {Do, Thao T.},
  title   = {{Z}arankiewicz's problem for semi-algebraic hypergraphs},
  journal = {Journal of Combinatorial Theory, Series A},
  volume  = {158},
  pages   = {621--642},
  year    = {2018},
  doi     = {10.1016/j.jcta.2018.04.007},
  eprint  = {1705.01979},
  archivePrefix = {arXiv},
  primaryClass = {math.CO}
}

@article{JanzerPohoata2024,
  author  = {Janzer, Oliver and Pohoata, Cosmin},
  title   = {On the {Z}arankiewicz problem for graphs with bounded {VC}-dimension},
  journal = {Combinatorica},
  volume  = {44},
  number  = {4},
  pages   = {839--848},
  year    = {2024},
  doi     = {10.1007/s00493-024-00095-2},
  eprint  = {2009.00130},
  archivePrefix = {arXiv},
  primaryClass = {math.CO}
}

@article{FranklPach1984,
  author  = {Frankl, Peter and Pach, J{\'a}nos},
  title   = {On disjointly representable sets},
  journal = {Combinatorica},
  volume  = {4},
  number  = {1},
  pages   = {39--45},
  year    = {1984},
  doi     = {10.1007/BF02579155}
}

@article{GeXuYipZhangZhao2026,
  author  = {Ge, Gennian and Xu, Zixiang and Yip, Chi Hoi and Zhang, Shengtong and Zhao, Xiaochen},
  title   = {The {Frankl}--{Pach} upper bound is not tight for any uniformity},
  journal = {Journal of Combinatorial Theory, Series A},
  volume  = {217},
  pages   = {106078},
  year    = {2026},
  doi     = {10.1016/j.jcta.2025.106078},
  eprint  = {2412.11901},
  archivePrefix = {arXiv},
  primaryClass = {math.CO}
}

@article{ChaoXuYipZhang2025,
  author  = {Chao, Ting-Wei and Xu, Zixiang and Yip, Chi Hoi and Zhang, Shengtong},
  title   = {Uniform set systems with small {VC}-dimension},
  journal = {International Mathematics Research Notices},
  volume  = {2025},
  number  = {17},
  pages   = {rnaf269},
  year    = {2025},
  doi     = {10.1093/imrn/rnaf269},
  eprint  = {2501.13850},
  archivePrefix = {arXiv},
  primaryClass = {math.CO}
}

@article{Laskowski1992,
  author  = {Laskowski, Michael C.},
  title   = {Vapnik-Chervonenkis classes of definable sets},
  journal = {Journal of the London Mathematical Society},
  volume  = {45},
  number  = {2},
  pages   = {377--384},
  year    = {1992},
  doi     = {10.1112/jlms/s2-45.2.377}
}

@article{ChudnovskySafra2008,
  author  = {Chudnovsky, Maria and Safra, Shmuel},
  title   = {The {E}rd{\H{o}}s--{H}ajnal conjecture for bull-free graphs},
  journal = {Journal of Combinatorial Theory, Series B},
  volume  = {98},
  number  = {6},
  pages   = {1301--1310},
  year    = {2008},
  doi     = {10.1016/j.jctb.2008.02.005}
}

@article{ChudnovskyFoxScottSeymourSpirkl2019,
  author  = {Chudnovsky, Maria and Fox, Jacob and Scott, Alex and Seymour, Paul and Spirkl, Sophie},
  title   = {Towards {E}rd{\H{o}}s--{H}ajnal for graphs with no 5-hole},
  journal = {Combinatorica},
  volume  = {39},
  number  = {5},
  pages   = {983--991},
  year    = {2019},
  doi     = {10.1007/s00493-019-3957-8},
  eprint  = {1803.03588},
  archivePrefix = {arXiv},
  primaryClass = {math.CO}
}

@article{ChudnovskyScottSeymourSpirkl2023,
  author  = {Chudnovsky, Maria and Scott, Alex and Seymour, Paul and Spirkl, Sophie},
  title   = {{E}rd{\H{o}}s--{H}ajnal for graphs with no 5-hole},
  journal = {Proceedings of the London Mathematical Society},
  volume  = {126},
  number  = {3},
  pages   = {997--1014},
  year    = {2023},
  doi     = {10.1112/plms.12504},
  eprint  = {2102.04994},
  archivePrefix = {arXiv},
  primaryClass = {math.CO}
}

@article{FoxNguyenScottSeymour2025,
  author  = {Fox, Jacob and Nguyen, Tung and Scott, Alex and Seymour, Paul},
  title   = {Induced subgraph density. {II}. Sparse and dense sets in cographs},
  journal = {European Journal of Combinatorics},
  volume  = {124},
  pages   = {104075},
  year    = {2025},
  doi     = {10.1016/j.ejc.2024.104075}
}

@article{NguyenScottSeymour2023V,
  author  = {Nguyen, Tung and Scott, Alex and Seymour, Paul},
  title   = {Induced subgraph density. {V}. {A}ll paths approach {E}rd{\H{o}}s--{H}ajnal},
  journal = {Advances in Combinatorics},
  year    = {2026},
  note    = {Accepted for publication},
  eprint  = {2307.15032},
  archivePrefix = {arXiv},
  primaryClass = {math.CO}
}

@article{PachSolymosi2001,
  author  = {Pach, J{\'a}nos and Solymosi, J{\'o}zsef},
  title   = {Crossing patterns of segments},
  journal = {Journal of Combinatorial Theory, Series A},
  volume  = {96},
  number  = {2},
  pages   = {316--325},
  year    = {2001},
  doi     = {10.1006/jcta.2001.3184}
}

@article{ChernikovStarchenko2021NIP,
  author  = {Chernikov, Artem and Starchenko, Sergei},
  title   = {Definable regularity lemmas for {NIP} hypergraphs},
  journal = {The Quarterly Journal of Mathematics},
  volume  = {72},
  number  = {4},
  pages   = {1401--1433},
  year    = {2021},
  doi     = {10.1093/qmath/haab011},
  eprint  = {1607.07701},
  archivePrefix = {arXiv},
  primaryClass = {math.LO}
}

@article{Forster2002,
  author  = {Forster, J{\"u}rgen},
  title   = {A linear lower bound on the unbounded error probabilistic communication complexity},
  journal = {Journal of Computer and System Sciences},
  volume  = {65},
  number  = {4},
  pages   = {612--625},
  year    = {2002},
  doi     = {10.1016/S0022-0000(02)00019-3}
}

@article{AlonMoranYehudayoff2017,
  author  = {Alon, Noga and Moran, Shay and Yehudayoff, Amir},
  title   = {Sign rank and the {V}apnik--{C}hervonenkis dimension},
  journal = {Sbornik: Mathematics},
  volume  = {208},
  number  = {12},
  pages   = {1724--1757},
  year    = {2017},
  doi     = {10.4213/sm8780},
  eprint  = {1503.07648},
  archivePrefix = {arXiv},
  primaryClass = {cs.LG}
}

@article{HuangJuZhou2026,
  author        = {Huang, Shenwei and Ju, Yiao and Zhou, Yidong},
  title         = {Erd{\H{o}}s--Hajnal beyond the five-vertex path},
  journal       = {arXiv preprint arXiv:2606.06258},
  year          = {2026},
  eprint        = {2606.06258},
  archivePrefix = {arXiv},
  primaryClass  = {math.CO},
  doi           = {10.48550/arXiv.2606.06258}
}

@article{YangYu2025,
  author        = {Yang, Tianchi and Yu, Xingxing},
  title         = {Maxmum Size of a Uniform Family with Bounded {VC}-dimension},
  journal       = {arXiv preprint arXiv:2508.14334},
  year          = {2025},
  eprint        = {2508.14334},
  archivePrefix = {arXiv},
  primaryClass  = {math.CO}
}

\end{document}